\documentclass[12pt]{amsart}

\usepackage{stmaryrd}
\usepackage{enumerate}
\usepackage{epsfig}
\usepackage{amsfonts,amsmath,amssymb,amsthm}

\begin{document} 

\newtheorem{theorem}{Theorem}[section]
\newtheorem{proposition}[theorem]{Proposition}
\newtheorem{lemma}[theorem]{Lemma}
\newtheorem{corollary}[theorem]{Corollary}
\newtheorem{definition}[theorem]{Definition}
\newtheorem{conjecture}[theorem]{Conjecture}
\newtheorem{example}[theorem]{Example}
\newtheorem{remark}[theorem]{Remark}

\newcommand{\cl}{\ensuremath{\prec}}
\newcommand{\cle}{\ensuremath{\preccurlyeq}}

\makeatletter
\newlength{\earraycolsep}
\setlength{\earraycolsep}{2pt}
\def\eqnarray{\stepcounter{equation}\let\@currentlabel%
\theequation
\global\@eqnswtrue\m@th
\global\@eqcnt\z@\tabskip\@centering\let\\\@eqncr
$$\halign to\displaywidth\bgroup\@eqnsel\hskip\@centering
$\displaystyle\tabskip\z@{##}$&\global\@eqcnt\@ne
\hskip 2\earraycolsep \hfil$\displaystyle{##}$\hfil
&\global\@eqcnt\tw@ \hskip 2\earraycolsep
$\displaystyle\tabskip\z@{##}$\hfil
\tabskip\@centering&\llap{##}\tabskip\z@\cr}
\makeatother

\title{Multi-triangulations as complexes of star~polygons}
\author{Vincent Pilaud \and Francisco Santos}

\thanks{This paper was written while the first author was visiting the second one via an internship agreement between the ENS and the University of Cantabria. Research of both authors was also funded by grant MTM2005-08618-C02-02 of the Spanish Ministry of Education and Science.}
\maketitle

\begin{abstract}
Maximal $(k+1)$-crossing-free graphs on a planar point set in convex position, that is, $k$-triangulations, have received attention in recent literature, with motivation coming from several interpretations of them.

We introduce a new way of looking at $k$-triangulations, namely as complexes of star polygons.
With this tool we give new, direct, proofs of the fundamental properties of $k$-triangulations, as well as some new results.
This interpretation also opens-up new avenues of research, that we briefly explore in the last section.
\end{abstract}

\section{Introduction}\label{sectionintroduction}

A \emph{multi-triangulation} of order $k$, or \emph{$k$-triangulation} of a convex $n$-gon is a maximal set of edges such that no $k+1$ of them mutually cross. 

\begin{example}
\rm
\label{exm:2k+2}
For $k=1$ these are simply triangulations. For $n\le 2k+1$, the complete graph $K_n$ on $n$ vertices does not contain $k+1$ mutually intersecting edges, and thus it is the unique $k$-triangulation of the $n$-gon. 
So, the first non-trivial case is $n=2k+2$. It is easy to check that there are $k+1$ $k$-triangulations of the $(2k+2)$-gon, each of them obtained from the complete graph $K_{2k+2}$ by deleting one of the diagonals $[i,i+k]$ (see Fig.~\ref{fig:2k+2}). See Lemma~\ref{lemma:2k+3} for a discussion of the next case, $n=2k+3$.

Figure~\ref{fig:2triang8points} shows another example; a $2$-triangulation of an octagon.
\end{example}

\begin{figure}
\centerline{\includegraphics[scale=1]{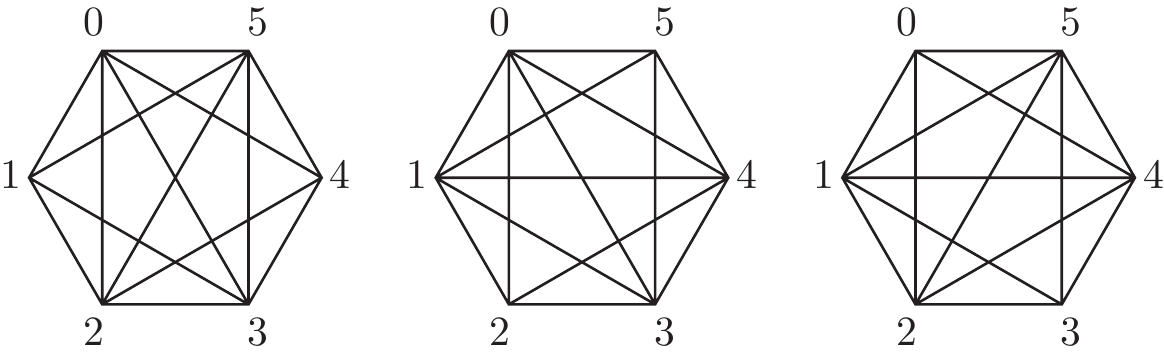}}
\caption{\small{The three $2$-triangulations of the hexagon.}}\label{fig:2k+2}
\end{figure}

\begin{figure}
\centerline{\includegraphics[scale=1]{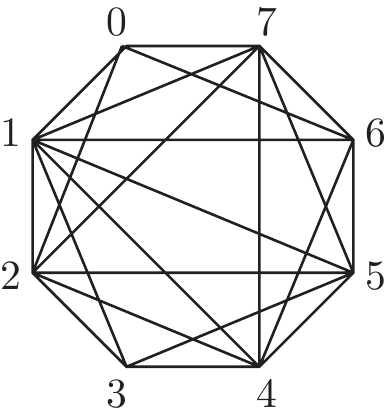}}
\caption{\small{A $2$-triangulation of the octagon.}}\label{fig:2triang8points}
\end{figure}

As far as we know, multi-triangulations first appear in the work of Capoyelas and Pach~\cite{cp-tttccp-92}, who prove that a $k$-triangulation of the $n$-gon cannot have more than $k(2n-2k-1)$ edges. 
Nakamigawa~\cite{n-gdfcp-00}, and independently Dress, Koolen and Moulton~\cite{dkm-lahp-02} then proved that all $k$ triangulations actually have that number of edges (for $n\ge 2k+1$). Both proofs use the concept of \emph{flip} between $k$-triangulations. As the name suggests, a flip creates one $k$-triangulation from another one, removing and inserting a single edge.
In fact, Nakamigawa~\cite{n-gdfcp-00} shows that essentially every edge of a $k$-triangulation can be flipped (our Corollary~\ref{graphflips}). 

The following list summarizes these and other nice properties of $k$-triangulations that have been proved in the literature:

\begin{theorem}[\cite{cp-tttccp-92,dkm-lahp-02, j-gt-03,n-gdfcp-00}]
\label{thm:intro}
\hfill
\begin{enumerate}[(a)]
\item All $k$-triangulations of a convex $n$-gon have the same number of edges, equal to $k(2n-2k-1)$~\cite{cp-tttccp-92,dkm-lahp-02,n-gdfcp-00}.
\item Any edge of length at least $k+1$ can be flipped and the graph of flips is regular and connected~\cite{dkm-lahp-02,n-gdfcp-00}.
\item The set of $k$-triangulations of the $n$-gon is enumerated by the same Catalan determinant counting families of $k$ mutually non-crossing Dyck paths~\cite{j-gt-03,j-gtdfssp-05}.
\item Any $k$-triangulation has at least $2k$ edges of length $k+1$~\cite{n-gdfcp-00} (this is the analogue of ``every triangulation has at least two ears").
\item There exists a ``deletion'' operation that allows one to obtain $k$-triangulations of an $n$-gon from those of an $(n+1)$-gon, and vice-versa~\cite{j-gt-03,n-gdfcp-00}.
\item The simplicial complex whose facets are the $k$-triangulations of an $n$-gon is a vertex-decomposable sphere of dimension $k(n-2k-1)-1$~\cite{j-gt-03}.
\end{enumerate}
\end{theorem}

In these statements the \emph{length} of an edge $[p_i,p_j]$ between $n$ points $p_1,\dots,p_n$ in convex position and labeled cyclically is defined as $\min\{|j-i|,|i-j|\} \mod n$. Only edges of length greater than $k$ are \emph{relevant}, since the rest cannot be part of a $(k+1)$-crossing (hence they show up in all $k$-triangulations).

It is interesting that $k$-triangulations admit several different interpretations: as particular cyclic split systems~\cite{dkm-2kn-01, dkm-4n-05}; as line arrangements in the hyperbolic plane~\cite{dkm-lahp-02}; and as certain fillings of triangular polyominoes~\cite{j-gtdfssp-05} (see also~\cite{k-gdidcffs-06, r-idsfmp-07}).
Another very close subject is the study of topological $k$-quasi-panar graphs~\cite{at-mneqpg-07}. That is, graphs without $k+1$ mutually intersecting edges, but where the edges are not forced to be straight line segments, they are only required to not intersect one another twice.

\medskip

In this paper we introduce a new way of looking at $k$-triangulations. Namely, as \emph{complexes of star polygons of type $\left\{\frac{2k+1}{k}\right\}$}.
If $p$ and $q$ are two coprime integers, a \emph{star polygon} of type $\{p/q\}$ is a polygon formed by connecting a set $V=\{s_j\;|\; j\in\mathbb{Z}_p\}$ of $p$ points of the unit circle with the set $E=\{[s_{j},s_{j+q}] \;|\; j\in\mathbb{Z}_p\}$ of edges of length $q$ (see~\cite[pp. 36-38]{c-ig-73},~\cite[ pp. 93-95]{c-rp-73} and Fig.~\ref{starpolygons}). The natural generalization of the triangle that is relevant for $k$-triangulations is:

\begin{definition}
A \emph{$k$-star} is a star polygon of type $\left\{\frac{2k+1}{k}\right\}$.
\end{definition}

\begin{figure}[!h]
\centerline{\includegraphics[scale=1]{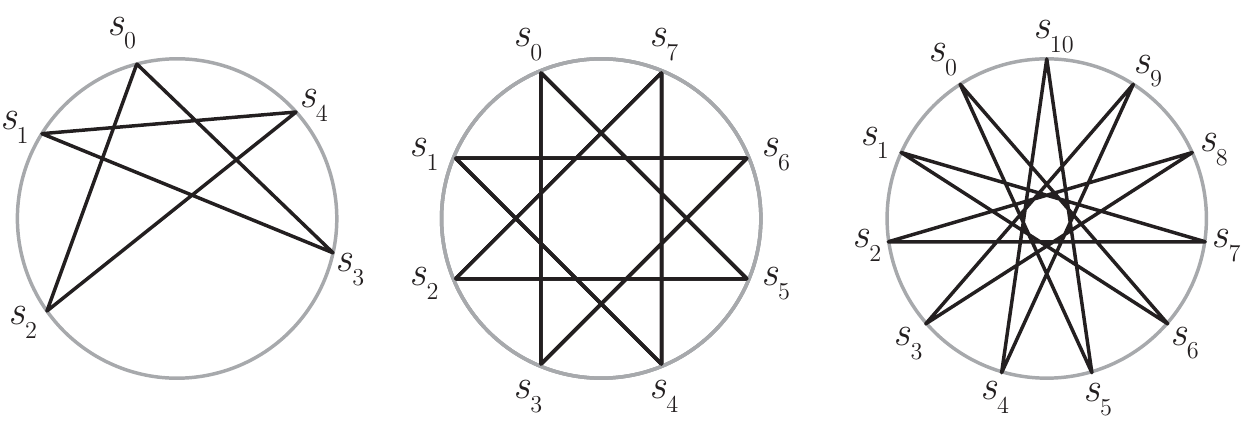}}
\caption{\small{Star polygons of type $\{5/2\}$, $\{8/3\}$ and $\{11/5\}$.
The first is a $2$-star and the last is a $5$-star.
}}\label{starpolygons}
\end{figure}

\newpage

Our main new result is:
\begin{theorem}
\label{thm:main}
Let $T$ be a $k$-triangulation of the $n$-gon (with $n\ge 2k+1$). Then
\begin{enumerate}
\item $T$ contains exactly $n-2k$ $k$-stars (Corollary~\ref{starsenumeration});
\item Each edge of $T$ belongs to zero, one or two $k$-stars, depending on whether its length is smaller, equal or greater than $k$ (Corollary~\ref{incidences});
\item Any common edge $f$ of two $k$-stars $R$ and $S$ of $T$ can be ``flipped" to another edge $e$ so that $T\triangle\{e,f\}$ is a $k$-triangulation. Moreover, the edges $e$ and $f$ depend only on $R\cup S$, not the rest of $T$ (Lemma~\ref{lemma:flip} and Corollary~\ref{graphflips}).
\end{enumerate}
\end{theorem}

But $k$-stars in $k$-triangulations are not only interesting ``per se''. Rather, we think they are the right way of looking at $k$-triangulations, in much the same way as triangles, instead of edges, are the right way of looking at triangulations. As evidence for this, in this paper we give new proofs of the basic properties of $k$-triangulations (parts (a), (b), (d) and (e) of Theorem~\ref{thm:intro}). Our proofs are direct and based on simple combinatorial properties of the mutual disposition of $k$-stars in a $(k+1)$-crossing-free graph (see Section~\ref{sectionstars}). In contrast, the way these basic results are proved in~\cite{dkm-lahp-02,j-gt-03,j-gtdfssp-05,n-gdfcp-00} is more indirect:

\begin{enumerate}
\item First, a ``deletion'' operation that relates $k$-triangulations of the $n$-gon and the $(n+1)$-gon is introduced. (We introduce this operation in Section~\ref{sectionflatinflat}).
\item Second, this operation is used to prove that $k$-triangulations admit flips (either in full generality or only for particular cases, as in~\cite{dkm-lahp-02}).
\item Finally, flips are used to show that every $k$-triangulation has the same number of edges as certain special ones constructed explicitly, namely $k(2n-2k-1)$.
\end{enumerate}

The structure of this paper is as follows: After introducing some notation in Section~\ref{sectionnotation} we discuss basic properties of $k$-stars in Section~\ref{sectionstars}. Section~\ref{sectioncomplexes} proves parts (1) and (2) of Theorem~\ref{thm:main}, as well as part (a) of Theorem~\ref{thm:intro}. Section~\ref{sectionflips} finishes the study of fundamental properties of $k$-triangulations, by introducing the graph of flips. We have to mention that another advantage of $k$-stars is that they allow for a much more explicit (and algorithmically better) way of understanding flips.

In Section~\ref{sectionears} we look at a particularly nice and simple family of $k$-triangulations. They are the analogue of ``triangulations with only two ears'' and admit several different characterizations. To emphasize that none of our proofs so far make use of the recursive operation relating $k$-stars of the $n$-gon and the $(n+1)$-gon we only introduce this operation in Section~\ref{sectionflatinflat}. The operation is almost literally the same one as in previous papers, but again it is more natural to look at it as the ``flattening'' (and, conversely, ``inflation'') of a single $k$-star.

Finally, we discuss in Section~\ref{sectionopen} further properties and questions about $k$-triangulations that may hopefully be easier to analyse using $k$-stars and which could be developed further in future papers. Among them:
\begin{itemize}
\item Dyck multi-paths: as mentioned in Theorem~\ref{thm:intro}, $k$-triangula\-tions of the $n$-gon are counted by the same determinant of Catalan numbers that counts certain families of non-crossing Dyck paths~\cite{j-gt-03,j-gtdfssp-05} (see also~\cite{k-gdidcffs-06, r-idsfmp-07}). An explicit bijection between these two combinatorial sets has only been found for $k=2$~\cite{e-btdp-06}.

\item multi-associahedron: we also said that the simplicial complex $\Delta_{n,k}$ whose facets are $k$-trianguations of the $n$-gon is a combinatorial sphere~\cite{j-gt-03}. It is natural to think that this sphere is polytopal, as happens for $k=1$ where it is the polar of the associahedron~\cite{l-at-89}.

\item rigidity: the number of edges of a $k$-triangulation of the $n$-gon is exactly that of a generically minimally rigid graph in dimension $2k$. We conjecture that all $k$-triangulations are minimally rigid in dimension $2k$ and prove it for $k=2$.

\item surfaces: regarding a $k$-triangulation $T$ as a complex of star polygons naturally defines a polygonal complex associated to it. This complex is an orientable surface with boundary. It seems interesting to study the action of flips on this surface. In particular, we can think of the fundamental group of the graph of flips as acting on the mapping class group of the surface.

\item chirotope: there is also a natural chirotope of rank $3$ defined on the set of $k$-stars of a $k$-triangulation. This is the analogue of the chirotopes (or pseudo-line arrangements) that Pocchiola and Vegter introduce on pseudo-triangulations~(\cite{pv-ot-94}, see also~\cite{pv-vc-96}).
\end{itemize}


\section{Notation}\label{sectionnotation}

Let $k$ and $n$ be two integers such that $k\ge 1$ and $n\ge 2k+1$.

Let $V_n$ be the \emph{set of vertices} of a convex $n$-gon, {\it i.e.}~any set of points on the unit circle, labelled counterclockwise by the cyclic set $\mathbb{Z}_n$.
All throughout the paper, we will refer to the points in $V_n$ by their labels to simplify notation.
For $u,v,w\in V_n$, we will write $u\cl v\cl w$ meaning that $u$, $v$ and $w$ are in counterclockwise order on the circle.
For any $u,v\in V_n$, let $\llbracket u,v\rrbracket$ denote the \emph{cyclic interval} $\{w\in V_n\;|\; u\cle w\cle v\}$.
The intervals $\rrbracket u,v\llbracket$, $\llbracket u,v\llbracket$ and $\rrbracket u,v\rrbracket$ are defined similarly.
Let $|u-v|=\min(|\llbracket u, v\llbracket|,|\llbracket v, u\llbracket|)$ be the \emph{cyclic distance} between $u$ and $v$.

For $u\ne v\in V_n$, let $[u,v]$ denote the \emph{edge} connecting the vertices $u$ and $v$. We say that $[u,v]$ is of \emph{length} $|u-v|$.
Let $E_n={V_n \choose 2}$ be the \emph{set of edges} of the complete graph on $V_n$.
Two edges $[u,v]$ and $[u',v']$ are said to \emph{cross} when the open segments $]u,v[$ and $]u',v'[$ intersect.
For $\ell\in\mathbb{N}$, an \emph{$\ell$-crossing} is a set of $\ell$ mutually intersecting edges.

\begin{definition}
A \emph{$k$-triangulation} of the $n$-gon is a maximal $(k+1)$-crossing-free subset of $E_n$.
\end{definition}

Obviously, an edge $[u,v]$ of $E_n$ can appear in a $(k+1)$-crossing only if $|u-v|>k$. We say that such an edge is \emph{$k$-relevant}. We say that an edge $[u,v]$ is a \emph{$k$-boundary} if $|u-v|=k$ and that it is \emph{$k$-irrelevant} if $|u-v|<k$.
Every $k$-triangulation of the $n$-gon consists of all the $kn$ $k$-irrelevant plus $k$-boundary edges and some $k$-relevant edges.

An \emph{angle} $\angle(u,v,w)$ of a subset $E$ of $E_n$ is a pair $\{[u,v],[v,w]\}$ of edges of $E$ such that $u\cl v\cl w$ and for all $t\in\rrbracket w,u\llbracket$, the edge $[v,t]$ is not in $E$. We call $v$ the \emph{vertex} of the angle $\angle(u,v,w)$. If $t\in\rrbracket w,u\llbracket$, we say that \emph{$t$ is contained} in $\angle(u,v,w)$, and that the edge $[v,t]$ is a \emph{bisector} of $\angle(u,v,w)$.
An angle is said to be \emph{$k$-relevant} if both its edges are either $k$-relevant or $k$-boundary edges.

As said in the introduction, a $k$-star is a set of edges of the form $\{[s_j,s_j+k] \;|\; j\in \mathbb{Z}_{2k+1} \}$, where $s_0\cl s_1 \cl \ldots \cl s_{2k} \cl s_0$ are cyclically ordered.
Observe that there are two natural cyclic orders on the vertices of a $k$-star $S$: the \emph{circle order}, defined as the cyclic order around the circle, and the \emph{star order}, defined as the cyclic order tracing the edges of $S$. More precisely, if $s_0,\ldots,s_{2k}$ are the vertices of $S$ cyclically ordered around the circle, we rename the vertices $r_i=s_{iq}$ to obtain the star order $r_0,\ldots,r_{2k}$.


\section{Mutual positions of $k$-stars}\label{sectionstars}

In this section, let $R$ and $S$ denote  two $k$-stars  of a $(k+1)$-crossing-free subset $E$ of $E_n$.
We study their mutual position.

\begin{lemma}\label{starangles}
\begin{enumerate}
\item Any angle of $S$ (or $R$) is also an angle of $E$ and is $k$-relevant.
\item For any vertex $t$ not in $S$, there is a unique angle $\angle(u,v,w)$ in $S$ that is bisected by $[v,t]$.
\end{enumerate}
\end{lemma}

\begin{proof}
Let $V=\{s_j\;|\; j\in\mathbb{Z}_{2k+1}\}$ denote the vertices of $S$ in star order.
Suppose that $E$ contains an edge $[s_j,t]$ where $j\in\mathbb{Z}_{2k+1}$ and $t\in\rrbracket s_{j+1},s_{j-1}\llbracket$. Then the set of edges $$\{[s_{j+1},s_{j+2}],[s_{j+3},s_{j+4}],\ldots,[s_{j-2},s_{j-1}],[s_j,t]\}$$ forms a $(k+1)$-crossing. Thus $\angle(s_{j-1},s_j,s_{j+1})$ is an angle of $E$.

Since any edge of $S$ separates the other vertices of $S$ into two parts of size $k-1$ and $k$, it is at least a $k$-boundary. Consequently, the angle $\angle(s_{j-1},s_j,s_{j+1})$ is $k$-relevant.
This finishes the proof of part (1). Part (2) is obvious from the definition of bisector.
\end{proof}

\begin{corollary}
$R$ and $S$ can not share any angle.
\end{corollary}

\begin{proof}
By the first point of the previous lemma, the knowledge of one angle $\angle(s_{j-1},s_j,s_{j+1})$ of $S$ permits the recovery of all the $k$-star $S$: the vertex $s_{j+2}$ is the unique vertex such that $\angle(s_j,s_{j+1},s_{j+2})$ is an angle of $E$ ({\it i.e.}~the first neighbour of $s_{j+1}$ after $s_j$ when moving clockwise), and so on.
\end{proof}

Since the number of edges of a $k$-star is $2k+1$, this corollary implies that $R$ and $S$ can not share more than $k$ edges. Note that, for example, the two $k$-stars of any $k$-triangulation of a $(2k+2)$-gon share exactly $k$ edges (see Figure~\ref{fig:2k+2}).

\begin{corollary}
For any edge $[u,v]$ of $E$, the number  of vertices of $S$ between $u$ and $v$ and the number  of vertices of $S$ between $v$ and $u$ are different.
\end{corollary}

\begin{proof}
Suppose that $S$ has the same number of vertices on both sides of an edge $f=[u,v]$. Since the number of vertices of $S$ is $2k+1$, one of the two vertices of $f$ is a vertex of $S$, say $u$. Then $v$ is contained in the angle of $S$ of vertex $u$, which implies that $f$ is not in $E$.
\end{proof}

Let $V$ be the set of vertices of the $k$-star $S$.
If $|\llbracket u,v\rrbracket\cap V|<|\llbracket v,u\rrbracket\cap V|$, then we say that $S$ lies on the \emph{positive side} of the oriented edge from $u$ to $v$ (otherwise we say that $S$ lies on the \emph{negative side} of the oriented edge from $u$ to $v$). 
The $k$-star $S$ is said to be \emph{contained in an angle} $\angle(u,v,w)$ of $E$ if it lies on the positive side of both the edges $[u,v]$ and $[v,w]$ oriented from $u$ to $v$ and from $v$ to $w$ respectively.

\begin{lemma}
Let $\angle(u,v,w)$ be an angle of $E$ containing the $k$-star $S$. Then
\begin{enumerate}[(i)]
\item either $v$ is a vertex of $S$ and $\angle(u,v,w)$ is an angle of $S$;
\item or $v$ is not a vertex of $S$ and $\angle(u,v,w)$ has a common bisector with an angle of $S$.
\end{enumerate}
\end{lemma}

\begin{proof}
Suppose first that $v$ is a vertex of $S$. Let $\angle(x,v,y)$ denote the angle of $S$ at vertex $v$. Since $\angle(u,v,w)$ contains $S$, we have $w\cle y\cl x\cle u$. But since $\angle(u,v,w)$ is an angle of $E$, we have $x=u$ and $y=w$, so that $\angle(u,v,w)$ is an angle of $S$.

Suppose now that $v$ is not a vertex of $S$. Then, by Lemma~\ref{starangles}~(2)
there exists a unique angle $\angle(x,y,z)$ of $S$ containing $v$. If $y\in\rrbracket u,v\llbracket$, then $\rrbracket u,v\llbracket$ contains all the $k+1$ vertices of $S$ between $y$ and $z$, which is not possible (because $S$ lies on the positive side of the edge $[u,v]$, oriented from $u$ to $v$). For the same reason, $y\not\in \rrbracket v,w \llbracket$.
If $y=u$ or $y=w$, then $[u,v]$ or $[v,w]$ is a bisector of $\angle(x,y,z)$, which contradicts Lemma~\ref{starangles}. 
Consequently, $\angle(u,v,w)$ contains $y$, so that $[v,y]$ is a common bisector of $\angle(u,v,w)$ and $\angle(x,y,z)$.
\end{proof}

In the following statement and the rest of the paper, a \emph{bisector} of $S$ is a bisector of an angle of $S$.

\begin{theorem}\label{common bisector}
Every pair of $k$-stars whose union is $(k+1)$-crossing-free have a unique common bisector.
\end{theorem}

\begin{proof}
Let $\{r_j\;|\; j\in\mathbb{Z}_{2k+1}\}$ denote the vertices of $R$ in star order. Note that for any $j\in\mathbb{Z}_{2k+1}$, if $S$ lies on the negative side of the edge $[r_{j-1},r_j]$ oriented from $r_{j-1}$ to $r_j$, then it lies on the positive side of the edge $[r_j,r_{j+1}]$ oriented from $r_j$ to $r_{j+1}$. Since $2k+1$ is odd, this implies that there exists $j\in\mathbb{Z}_{2k+1}$ such that $S$ lies on the positive side of both the edges $[r_{j-1},r_j]$ and $[r_j,r_{j+1}]$ oriented from $r_{j-1}$ to $r_j$ and from $r_j$ to $r_{j+1}$, respectively. That is, in the angle $\angle(r_{j-1},r_j,r_{j+1})$.The previous lemma thus ensures the existence of a common bisector.

Suppose now that we have two different common bisectors $e$ and $f$ of $R$ and $S$, and let $\{r_j\;|\; j\in\mathbb{Z}_{2k+1}\}$ and $\{s_j\;|\; j\in\mathbb{Z}_{2k+1}\}$ denote the vertices of $R$ and $S$ in star order labeled so that $e=[r_0,s_0]$. Let $a,b\in\mathbb{Z}_{2k+1}$ such that $f=[r_a,s_b]$.
Note that certainly, $a\ne0$, $b\ne0$, and $a$ and $b$ have the same parity. By symmetry, we can assume that $a=2\alpha$, $b=2\beta$ with $1\le\beta\le\alpha\le k$. But then the set
$$\{[r_{2i},r_{2i+1}]\;|\; 0\le i\le \alpha-1\}\cup\{[s_{2j},s_{2j+1}]\;|\; \beta\le j\le k\}$$
forms a $(k+1+\alpha-\beta)$-crossing, and $k+1+\alpha-\beta\ge k+1$. This proves uniqueness.\end{proof}

\begin{figure}
\centerline{\includegraphics[scale=1]{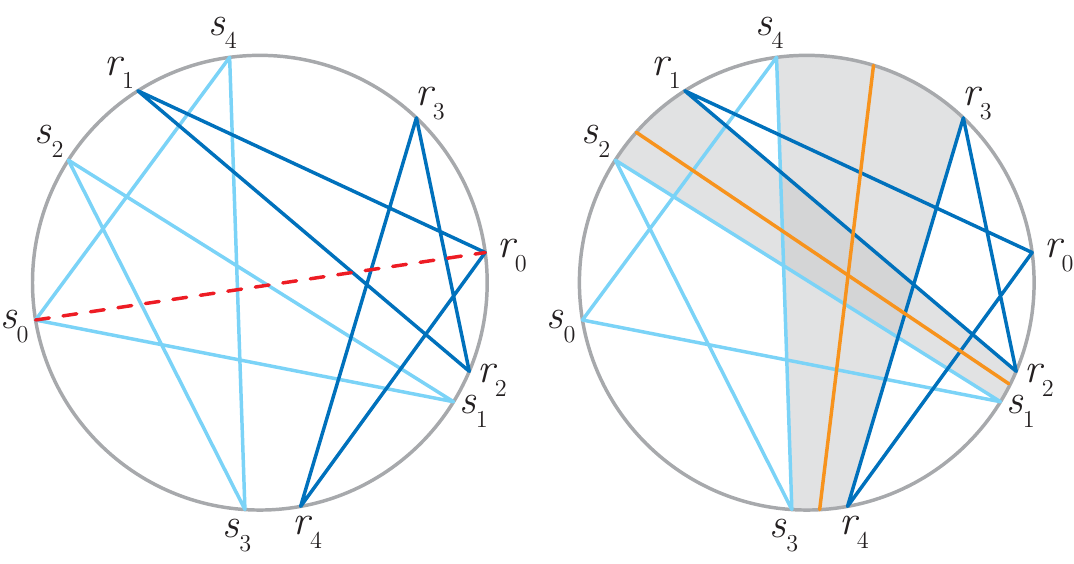}}
\caption{\small{The common bisector of two $2$-stars (left) and a $2$-crossing that crosses it (right).}}\label{combisector}
\end{figure}

In the following lemmas, $\{r_j\;|\; j\in\mathbb{Z}_{2k+1}\}$ and $\{s_j\;|\; j\in\mathbb{Z}_{2k+1}\}$ denote the vertices of $R$ and $S$, in star order, and with $e=[r_0,s_0]$ being the common bisector of $R$ and $S$.

\begin{lemma}\label{paralleledges}
For every $1\le i\le k$, $r_{2i-1}\in\rrbracket r_0,s_{2i}\rrbracket$ and $s_{2i-1}\in\rrbracket s_0,r_{2i}\rrbracket$.
In particular, for every $i\in\mathbb{Z}_{2k+1}$ the edges $[r_i,r_{i+1}]$ and $[s_i,s_{i+1}]$ do not cross.
\end{lemma}

\begin{proof}
Suppose that there exists $1\le i\le k$ such that $r_{2i-1}\in\rrbracket s_{2i},s_0\llbracket$ or $s_{2i-1}\in\rrbracket r_{2i},r_0\llbracket$. Let $\gamma$ be the highest such integer, and assume for example that $r_{2\gamma-1}\in\rrbracket s_{2\gamma},s_0\llbracket$. Then the definition of $\gamma$ ensures that
$s_0\cl s_{2\gamma+1}\cle r_{2\gamma+2}\cl r_{2\gamma-2}$
so that the set
$$\{[r_{2i},r_{2i+1}]\;|\; 0\le i\le\gamma-1\}\cup\{[s_{2j},s_{2j+1}]\;|\; \gamma\le j\le k\}$$
forms a $(k+1)$-crossing.
\end{proof}

The previous lemma can be read as saying that corresponding edges of $R$ and $S$ are parallel. It is easy to see that $k$ of these $2k+1$ pairs of parallel edges, the ones of the form $([r_{2i-1},r_{2i}],[s_{2i-1},s_{2i}])$, with $1\le i\le k$, separate $R$ from $S$, meaning that $R$ and $S$ lie on opposite sides of both. The next lemma says that any $k$-crossing that, in turn, crosses the common bissector $e=[r_0,s_0]$ has one edge parallel to and between each such pair (see Fig.~\ref{combisector}).

\begin{lemma}\label{sandwich}
Let $F$ be a $k$-crossing of $E$ such that all its edges cross $e=[r_0,s_0]$. Let $f_1=[x_1,y_1],\ldots,f_k=[x_k,y_k]$ denote the edges of $F$, with $r_0\cl x_1\cl \ldots\cl x_k\cl s_0\cl y_1\cl \ldots\cl y_k\cl r_0$.
Then for any $1\le i\le k$, we have $x_i\in\llbracket r_{2k-2i+1},s_{2k-2i+2}\rrbracket$ and $y_i\in\llbracket s_{2k-2i+1},r_{2k-2i+2}\rrbracket$.
\end{lemma}

\begin{proof}
Suppose that there exists $1\le i\le k$ such that $r_0\cl x_i\cl r_{2k-2i+1}$ and let
$$\ell=\max\{1\le i\le k \;|\; r_0\cl x_i\cl r_{2k-2i+1}\}.$$
If $\ell=k$, then the set $\{e_1,\ldots,e_k,[r_0,r_1]\}$ is a $(k+1)$-crossing of $E$, thus we assume that $\ell<k$. In order for $$\{e_1,\ldots,e_\ell\}\cup\{[r_0,r_1],\ldots,[r_{2k-2\ell},r_{2k-2\ell+1}]\}$$ not to be a $(k+1)$-crossing, we have $r_{2k-2\ell}\cle y_\ell\cl r_0$, so that $r_{2k-2\ell}\cl y_{\ell+1}\cl r_0$. But the definition of $\ell$ implies that $r_{2k-2\ell+1}\cl r_{2k-2\ell-1}\cle x_{\ell+1}\cl s_0$, so that
the set $$\{[r_{2k-2\ell},r_{2k-2\ell+1}],\ldots,[r_{2k-2},r_{2k-1}],[r_{2k},r_0]\}\cup\{e_{\ell+1},\ldots,e_k\}$$ is a $(k+1)$-crossing of $E$.

By symmetry, the lemma is proved.
\end{proof}

The following lemma is at the heart of the concept of flips, that will be studied further in Section~\ref{sectionflips}. As usual, we use the symbol $\triangle$ for the symmetric difference.

\begin{lemma}\label{commonedge}
Let $f$ be a common edge of $R$ and $S$. 
Let $\{r_j\;|\; j\in\mathbb{Z}_{2k+1}\}$ and $\{s_j\;|\; j\in\mathbb{Z}_{2k+1}\}$ be the vertices of $R$ and $S$, in star order, and with $e=[r_0,s_0]$ being the common bisector of $R$ and $S$.
Then
\begin{enumerate}[(i)]
\item there exists $1\le i\le k$ such that $f=[r_{2i-1},r_{2i}]=[s_{2i},s_{2i-1}]$;
\item $E\triangle\{e,f\}$ is a $(k+1)$-crossing-free subset of $E_n$;
\item the vertices $$s_0,\ldots,s_{2i-2},s_{2i-1}=r_{2i},r_{2i+1},\ldots,r_{2k},r_0$$
(resp. $r_0,\ldots,r_{2i-2},r_{2i-1}=s_{2i},s_{2i+1},\ldots,s_{2k},s_0$) are the vertices of a $k$-star $X$ (resp. $Y$) of $E\triangle\{e,f\}$, in star order;
\item $X$ and $Y$ share the edge $e$ and their common bisector is $f$.
\end{enumerate}
\end{lemma}

\begin{figure}
\centerline{\includegraphics[scale=1]{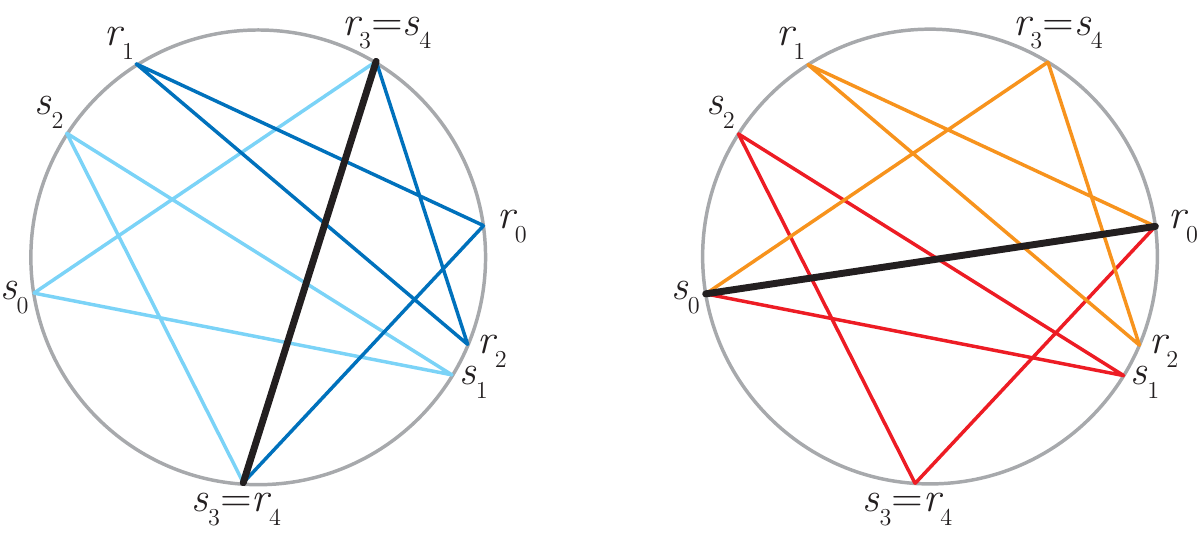}}
\caption{\small{The flip of an edge.}}\label{fig:flip}
\end{figure}

\begin{proof}
Let $u$ and $v$ denote the vertices of $f$.
Lemma~\ref{paralleledges} ensures that $\{r_0,s_0\}\cap\{u,v\}=\emptyset$ so that we can assume $r_0\cl u\cl s_0\cl v\cl r_0$. Consequently, there exists $1\le i,j\le k$ such that $u=r_{2i-1}=s_{2j}$ and $v=r_{2i}=s_{2j-1}$. Suppose that $i>j$. Then according to Lemma~\ref{paralleledges}, we have $r_0\cl r_{2i-1}\cle s_{2i}\cl s_{2j}=r_{2i-1}$ which is impossible. We obtain that $i=j$ and $f=[r_{2i-1},r_{2i}]=[s_{2i},s_{2i-1}]$.

Lemma~\ref{sandwich} then proves that any $k$-crossing of $T$ that prevent $e$ being in $T$ contains $f$, so that $T\triangle\{e,f\}$ is $(k+1)$-crossing-free.

Let $L$ be the list of vertices $(s_0,\ldots,s_{2i-2},r_{2i},\ldots,r_{2k},r_0)$. Between two consecutive elements of $L$ lie exactly $k-1$ others points of $L$ (for the circle order). This implies that $L$ is in star order. The point (iii) thus follows from the fact that any edge connecting two consecutive points of $L$ is in $E\triangle\{e,f\}$.

The edge $e$ is clearly common to $X$ and $Y$. The edge $f$ is a bisector of both angles $\angle(r_{2i-2},r_{2j-1},s_{2i+1})$ and $\angle(s_{2i-2},s_{2j-1},r_{2i+1})$, so that it is the common bisector of $X$ and $Y$.
\end{proof}


\section{$k$-triangulations as complexes of $k$-stars}\label{sectioncomplexes}

In this section, $T$ is a $k$-triangulation of the $n$-gon, {\it i.e.}~a maximal $(k+1)$-crossing-free subset of $E_n$.

\begin{theorem}\label{angle}
Any $k$-relevant angle of $T$ belongs to a unique $k$-star contained in $T$.
\end{theorem}

\begin{proof}
In this proof, we need the following definition (see Fig.~\ref{farther}): let $\angle(u,v,w)$ be a $k$-relevant angle of $T$ and let $e$ and $f$ be two edges of $T$ that \emph{intersect} $\angle(u,v,w)$ ({\it i.e.}~that intersect both $[u,v]$ and $[v,w]$). If $a,b,c$ and $d$ denote their vertices such that $u\cl a\cl v\cl b\cl w$ and $u\cl c\cl v\cl d\cl w$, then we say that $e=[a,b]$ is \emph{$v$-farther} than $f=[c,d]$ if $u\cl a\cle c\cl v\cl d\cle b\cl w$. Let $E$ and $F$ be two $(k-1)$-crossings that \emph{intersect} $\angle(u,v,w)$. Let their edges be labelled $e_1=[a_1,b_1],e_2=[a_2,b_2],\ldots,e_{k-1}=[a_{k-1},b_{k-1}]$ and $f_1=[c_1,d_1],f_2=[c_2,d_2],\ldots,f_{k-1}=[c_{k-1},d_{k-1}]$ such that $u\cl a_1\cl a_2\cl \ldots\cl a_{k-1}\cl v\cl b_1\cl b_2\cl \ldots\cl b_{k-1}\cl w$ and $u\cl c_1\cl c_2\cl \ldots\cl c_{k-1}\cl v\cl d_1\cl d_2\cl \ldots\cl d_{k-1}\cl w$. Then we say that $E$ is \emph{$v$-farther} than $F$ if $e_i$ is $v$-farther than $f_i$ for every $1\le i\le k-1$. We say that $E$ is \emph{$v$-maximal} if it does not exist any $(k-1)$-crossing intersecting $\angle(u,v,w)$ and $v$-farther than $E$.

\begin{figure}
\centerline{\includegraphics[scale=1]{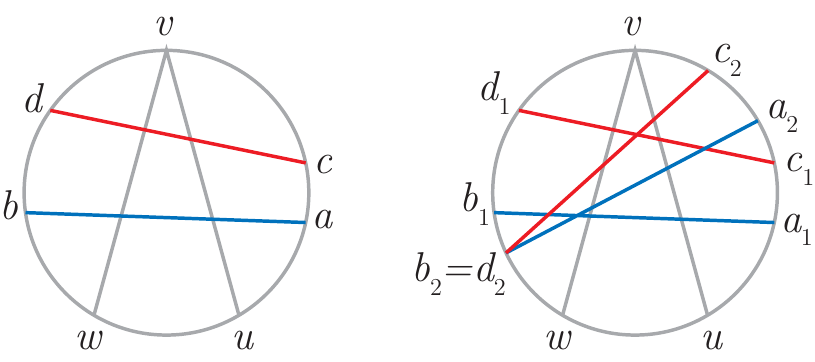}}
\caption{\small{$[a,b]$ is $v$-farther than $[c,d]$ (left) and $\{[a_i,b_i]\}$ is $v$-farther than $\{[c_i,d_i]\}$ (right).}}\label{farther}
\end{figure}

Let $\angle(u,v,w)$ be a $k$-relevant angle of $T$. 
If the edge $[u,v+1]$ is in $T$, then the angle $\angle(v+1,u,v)$ is a $k$-relevant angle of $T$ and, if it is contained in a $k$-star $S$ of $T$, then so is $\angle(u,v,w)$. Moreover, if $n>2k+1$ ($n=2k+1$ is a trivial case), $T$ can not contain all the edges $\{[u+i,v+i]\;|\; i=0\ldots n-1\}$ and $\{[u+i,v+i+1]\;|\; i=0\ldots n-1\}$. Consequently, we can assume that $[u,v+1]$ is not in $T$.

Thus we have a $k$-crossing $E$ of the form $e_1=[a_1,b_1],\ldots,e_k=[a_k,b_k]$ with $u\cl a_1\cl \ldots\cl a_k\cl v+1$ and $v+1\cl b_1\cl \ldots\cl b_k\cl u$. Since $[u,v]\in T$, $a_k=v$ and since $\angle(u,v,w)$ is an angle, $v+1\cl b_k\cle w$. Consequently, $\{e_1,\ldots,e_{k-1}\}$ forms a $(k-1)$-crossing intersecting $\angle(u,v,w)$ and we can assume that it is $v$-maximal (see Fig.~\ref{twosteps}~(a)). We will prove that the edges $[u,b_1], [a_1,b_2],\ldots, [a_{k-2},b_{k-1}],[a_{k-1},w]$ are in $T$ such that the points $u$, $a_1,\ldots,a_{k-1}$, $v$, $b_1,\ldots,b_{k-1}$, $w$ are the vertices of a $k$-star of $T$ containing the angle $\angle(u,v,w)$. To get this result, we use two steps: first we prove that $\angle(a_1,b_1,u)$ is an angle of $T$, and then we prove that the edges $e_2,\ldots,e_{k-1},[v,w]$ form a $(k-1)$-crossing intersecting $\angle(a_1,b_1,u)$ and $b_1$-maximal (so that we can reiterate the argument).

\medskip
\noindent\textsc{First step.} (see Fig.~\ref{twosteps}~(b))

Suppose that $[u,b_1]$ is not in $T$. Thus we have a $k$-crossing $F$ that prevents the edge $[u,b_1]$. Let $f_1=[c_1,d_1],\ldots,f_k=[c_k,d_k]$ denote its edges with $u\cl c_1\cl \ldots\cl c_k\cl b_1$ and $b_1\cl d_1\cl \ldots\cl d_k\cl u$.

Note first that $v\cl d_k\cle w$. Indeed, if it is not the case, then $d_k\in\rrbracket w,u\llbracket$ and $c_k\ne v$, because $\angle(u,v,w)$ is an angle. Thus either $c_k\in\rrbracket u,v\llbracket$ and then $F\cup\{[u,v]\}$ forms a $(k+1)$-crossing, or $c_k\in\rrbracket v,b_1\llbracket$ and then $E\cup\{[c_k,d_k]\}$ forms a $(k+1)$-crossing. Consequently, we have $b_1\cl d_1\cl \ldots\cl d_{k-1}\cl w$.

Let $\ell=\max\{1\le i\le {k-1}\;|\; b_i\cl d_i\cl w\}$. Then for any $1\le i\le\ell$, since $\{e_1,\ldots,e_i\}\cup\{f_i,\ldots,f_k\}$ does not form a $(k+1)$-crossing, we have $u\cl c_i\cle a_i$. Thus for any $1\le i\le\ell$, $u\cl c_i\cle a_i\cl v\cl b_i\cl d_i\cl w$, so that $f_i$ is $v$-farther than $e_i$. Furthermore, we have $u\cl c_1\cl \ldots\cl c_\ell\cl a_{\ell+1}\cl\ldots\cl a_{k-1}\cl v\cl d_1\cl\ldots\cl d_\ell\cl b_{\ell+1}\cl \ldots\cl b_{k-1}\cl w$. Consequently, we get a $(k-1)$-crossing $\{f_1,\ldots,f_\ell,e_{\ell+1},\ldots,e_{k-1}\}$ which is $v$-farther than $\{e_1,\ldots,e_{k-1}\}$; this contradicts the definition of $\{e_1,\ldots,e_{k-1}\}$. Thus we obtain $[u,b_1]\in T$.

Suppose now that $\angle(a_1,b_1,u)$ is not an angle of $T$. Then there exists $a_0\in \rrbracket u,a_1\llbracket$ such that $[b_1,a_0]\in T$. But then the $(k-1)$-crossing $\{[a_0,b_1],e_2,\ldots,e_{k-1}\}$ is $v$-farther than $\{e_1,\ldots,e_{k-1}\}$. This implies that $\angle(a_1,b_1,u)$ is an angle of $T$.
\begin{figure}
\centerline{\includegraphics[scale=1]{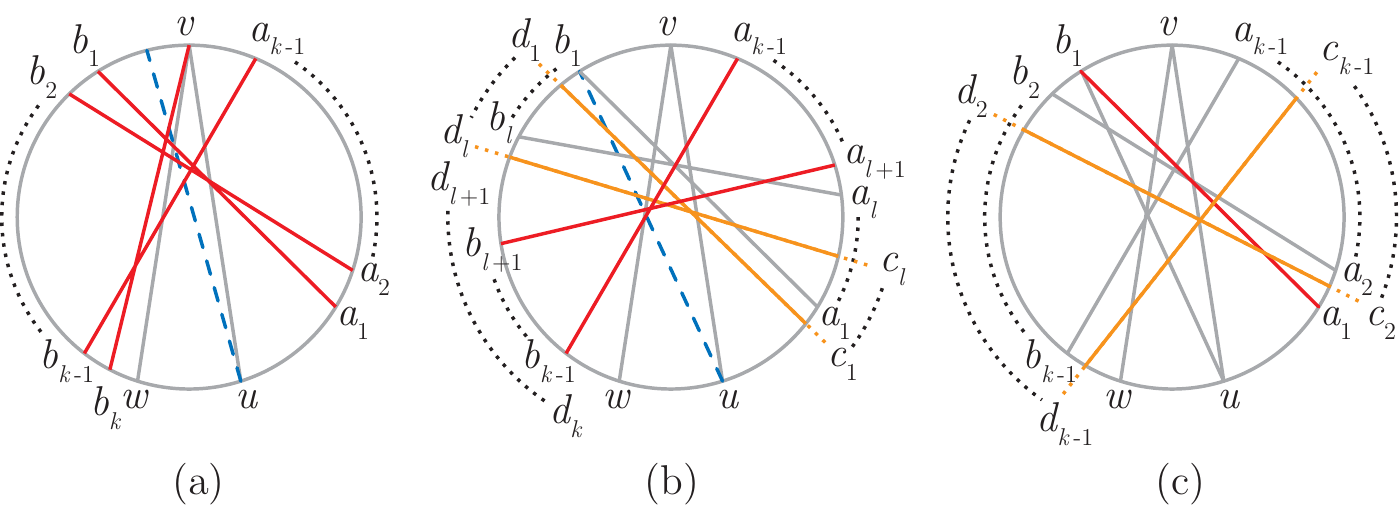}}
\caption{\small{(a) the $k$-crossing $E$; (b) First step: $[u,b_1]\in T$; (c) Second step: $\{e_2,\ldots,e_{k-1},[v,w]\}$ is $b_1$-maximal.}}\label{twosteps}
\end{figure}

\medskip
\noindent\textsc{Second step.} (see Fig.~\ref{twosteps}~(c))

Let $F$ be a $(k-1)$-crossing intersecting $\angle(a_1,b_1,u)$ and $b_1$-farther than the $(k-1)$-crossing $\{e_2,\ldots,e_{k-1},[v,w]\}$. Let $f_2=[c_2,d_2],\ldots,f_k=[c_k,d_k]$ denote its edges, with $a_1\cl c_2\cl \ldots\cl c_k\cl b_1\cl d_2\cl \ldots\cl d_k\cl u$. 

Note first that $b_k\cle d_k\cle w$. Indeed, if it is not the case, then $d_k\in\rrbracket w,u\llbracket$ and $c_k\ne v$, because $\angle(u,v,w)$ is an angle. Thus either $c_k\in\rrbracket a_1,v\llbracket$ and then $F\cup\{[u,v],e_1\}$ forms a $(k+1)$-crossing, or $c_k\in\rrbracket v,b_1\llbracket$ and then $E\cup\{[c_k,d_k]\}$ forms a $(k+1)$-crossing. Consequently, we have $b_1\cl d_2\cl\ldots\cl d_{k-1}\cl w$.

Furthermore, for any $2\le i\le k-1$, $f_i$ is $\angle(a_1,b_1,u)$-farther than $e_i$, so that $a_1\cl c_i\cle a_i\cl b_1\cl b_i\cle d_i\cl u$. In particular, $a_1\cl c_{k-1}\cle a_{k-1}\cl v$ and we get $u\cl a_1\cl c_2\cl\ldots\cl c_{k-1}\cl v\cl b_1\cl d_2\cl \ldots\cl d_{k-1}\cl w$. Consequently, the $(k-1)$-crossing $\{e_1,f_2,\ldots,f_{k-1}\}$ is $v$-farther than $\{e_1,\ldots,e_{k-1}\}$, which is a contradiction.
\end{proof}

The following easy consequence of Theorem~\ref{angle} justifies the title of this paper.

\begin{corollary}\label{incidences}
Let $e$ be an edge of $T$.
\begin{enumerate}[(i)]
\item If $e$ is a $k$-relevant edge, then it belongs to exactly two $k$-stars of $T$ (one on each side);
\item If $e$ is a $k$-boundary edge, then it belongs to exactly one $k$-star of $T$ (on its ``inner" side); \item If $e$ is a $k$-irrelevant edge, then it does not belong to any $k$-star of $T$.
\end{enumerate}
\end{corollary}

\begin{corollary}\label{findingstars}
\begin{enumerate}
\item For any $k$-star $S$ in $T$ and for any vertex $r$ not in $S$ there is a unique $k$-star $R$ in $T$ such that $r$ is a vertex of the common bisector of $R$ and $S$.
\item Any $k$-relevant edge which is not in $T$ is the common bisector of a unique pair of $k$-stars of $T$.
\end{enumerate}
\end{corollary}

\begin{proof}
Let $\angle(u,s,v)$ be the unique angle of $S$ which contains $r$. Let $\angle(x,r,y)$ be the unique angle of $T$ of vertex $r$ which contains $s$. According to Theorem~\ref{angle}, the angle $\angle(x,r,y)$ belongs to a unique $k$-star $R$. The common bisector of $R$ and $S$ is $[r,s]$ and $R$ is the only such $k$-star of $T$.

Let $e=[r,s]$ be a $k$-relevant edge, not in $T$. Let $\angle(x,r,y)$ (resp. $\angle(u,s,v)$) denote the unique angle of $T$ of vertex $r$ (resp. $s$) which contains $s$ (resp. $r$). According to Theorem~\ref{angle}, the angle $\angle(x,r,y)$ (resp. $\angle(u,s,v)$) belongs to a unique $k$-star $R$ (resp. $S$). The common bisector of $R$ and $S$ is $[r,s]$ and $(R,S)$ is the only such couple of $k$-stars of $T$.
\end{proof}

Observe that parts (1) and (2) of this lemma give bijections from 
\begin{enumerate}
\item ``vertices not used in the $k$-star $S$ of $T$" and ``$k$-stars of $T$ different from $S$";
\item ``$k$-relevant edges not used in $T$" and ``pairs of $k$-stars of $T$".
\end{enumerate}
From any of these two bijections, and using Corollary~\ref{incidences} for double counting, it is easy to derive the number of $k$-stars and of $k$-relevant edges in $T$:

\begin{corollary}\label{starsenumeration}
\begin{enumerate}
\item Any $k$-triangulation of the $n$-gon contains exactly $n-2k$ $k$-stars, $k(n-2k-1)$ $k$-relevant edges and $k(2n-2k-1)$ edges.
\item The $k$-triangulations  are exactly the $(k+1)$-crossing-free subsets of $E_n$ of cardinality $k(2n-2k-1)$.
\end{enumerate}
\end{corollary}

Similarly, we prove the following result:

\begin{corollary}
\label{cor:numberofstars}
Let $[u,v]$ be a $k$-relevant or $k$-boundary edge of $T$. Then the number of $k$-stars of $T$ on the positive side of the oriented edge from $u$ to $v$ equals $|\llbracket v,u\llbracket|-k$.
\end{corollary}

\begin{proof}
If $|\llbracket v,u\llbracket|=k$, then any $k$-star of $T$ has certainly more vertices in $\llbracket u,v\rrbracket$ than in $\llbracket v,u\rrbracket$, so that any $k$-star of $T$ lies on the negative side of $[u,v]$.

Suppose now that $|\llbracket v,u\llbracket|>k$. Let $R$ be the $k$-star of $T$ containing $[u,v]$ and lying on the positive side of it. According to the bijection~(1) given by Corollary~\ref{findingstars}, we only have to prove that any $k$-star $S$ of $T$ (distinct from $R$) lies on the positive side of $[u,v]$ if and only if the vertex of $S$ of the common bisector of $R$ and $S$ lies in $\rrbracket v,u\llbracket$. But this is an easy  consequence of Lemma~\ref{paralleledges}.
\end{proof}

\begin{example}
\rm
To illustrate the results of this section, Figure~\ref{fig:2triang8pointsstars} shows the $2$-stars contained in the $2$-triangulation $T$ of the octagon of Figure~\ref{fig:2triang8points}. There are four $2$-stars and six $2$-relevant edges (Corollary~\ref{starsenumeration}). Every $2$-relevant edge is contained in two $2$-stars of $T$ (Corollary~\ref{incidences}) and each of the twenty $2$-relevant angles of $T$ is contained in exactly one $2$-star of $T$ (Theorem~\ref{angle}).
\end{example}

\begin{figure}
\centerline{\includegraphics[scale=1]{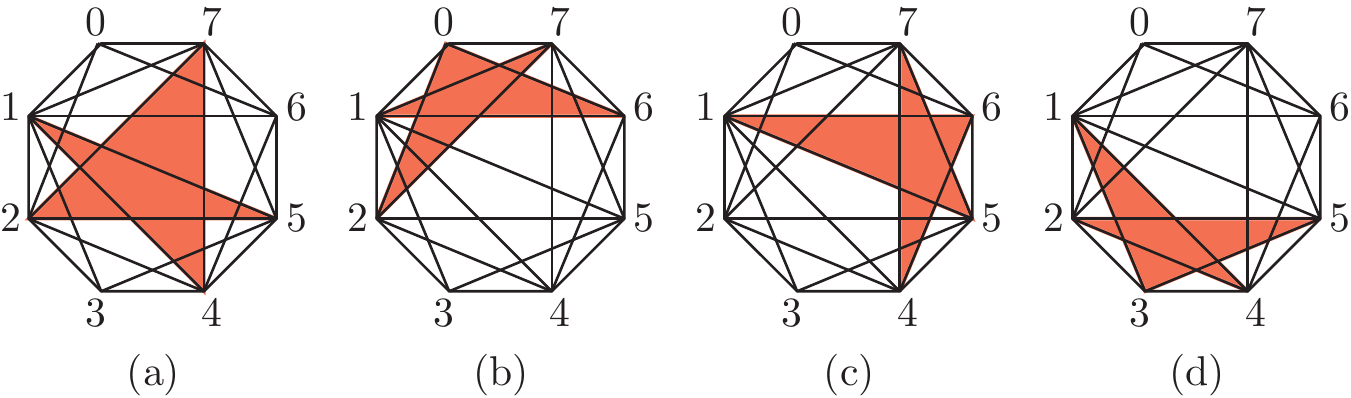}}
\caption{\small{The four $2$-stars in the $2$-triangulation of the octagon of Figure~\ref{fig:2triang8points}.}}\label{fig:2triang8pointsstars}
\end{figure}


\section{The graph of flips}\label{sectionflips}

Let $T$ be a $k$-triangulation of the $n$-gon, let $f$ be a $k$-relevant edge of $T$, let $R$ and $S$ be the two $k$-stars of $T$ containing $f$ (Corollary~\ref{incidences}), and let $e$ be the common bisector of $R$ and $S$. Remember that we proved (Lemma~\ref{commonedge}) that $T\triangle\{e,f\}$ is a $(k+1)$-crossing-free subset of $E_n$. Observe moreover that $T\triangle\{e,f\}$ is maximal: this follows from Corollary~\ref{starsenumeration}, but also from the fact that if $T\triangle\{e,f\}$ is properly contained in a $k$-triangulation $\widetilde{T}$, then $\widetilde{T}\triangle\{e,f\}$ is $(k+1)$-crossing-free and properly contains $T$. 

Thus, $T\triangle\{e,f\}$ is a $k$-triangulation of the $n$-gon.

\begin{lemma}\label{lemma:flip}
$T$ and $T\triangle\{e,f\}$ are the only two $k$-triangulations of the $n$-gon contaning $T\setminus\{f\}$.
\end{lemma}

\begin{proof}
Let $e'$ be any edge of $E_n\setminus T$ distinct from $e$. Let $R'$ and $S'$ be the two $k$-stars with common bisector $e'$ (Lemma~\ref{findingstars}~(2)). We can assume that $R'$ does not contain $f$ and then $R'\cup\{e'\}$ is contained in $T\triangle\{e',f\}$ and forms a $(k+1)$-crossing.
\end{proof}

We say that we obtain the $k$-triangulation $T\triangle\{e,f\}$ from the $k$-triangulation $T$ by \emph{flipping} the edge $f$. 

Let $G_{n,k}$ be the \emph{graph of flips} on the set of $k$-triangulations of the $n$-gon, {\it i.e.}~the graph whose vertices are the $k$-triangulations of the $n$-gon and whose edges are the pairs of $k$-triangulations related by a flip.
It follows from Corollary~\ref{starsenumeration} and Lemma~\ref{lemma:flip} that $G_{n,k}$ is regular of degree $k(n-2k-1)$: every $k$-relevant edge of $T$ can be flipped, in a unique way. In this section, we prove the connectivity of this graph and bound its diameter.

\medskip
Note that $e$ and $f$ necessarily cross.
In particular, if $e=[\alpha,\beta]$ and $f=[\gamma,\delta]$, with $0\cle\alpha\cl\beta\cle n-1$ and $0\cle\gamma\cl\delta\cle n-1$, then
\begin{itemize}
\item either $0\cle\alpha\cl\gamma\cl\beta\cl\delta\cle n-1$ and the flip is said to be \emph{slope-decreasing},
\item or $0\cle\gamma\cl\alpha\cl\delta\cl\beta\cle n-1$ and the flip is said to be \emph{slope-increasing}.
\end{itemize}
We define a partial order on the set of $k$-triangulations of the $n$-gon as follows: for two $k$-triangulations $T$ and $T'$, we say $T<T'$ if and only if there exists a sequence of slope-increasing flips from $T$ to $T'$.
That this is indeed a partial order follows from the fact that each slope increasing flip increases the \emph{total slope} of a $k$-triangulation, where the slope of an edge $[u,v]$ is defined as $u+v$ (with the sum taken in $\mathbb{N}$, not in $\mathbb{Z}_{n}$)  and the total slope of a $k$-triangulation is the sum of slopes of its edges.

Let $T_{n,k}^{\min}$ be the $k$-triangulation of the $n$-gon whose set of $k$-relevant edges is $\{[i,j]\;|\;i\in\llbracket 0,k-1\rrbracket\;\mathrm{and}\;j\in\llbracket i+k,i-k\rrbracket\}$ (see Fig.~\ref{min}).

\begin{figure}
\centerline{\includegraphics[scale=1]{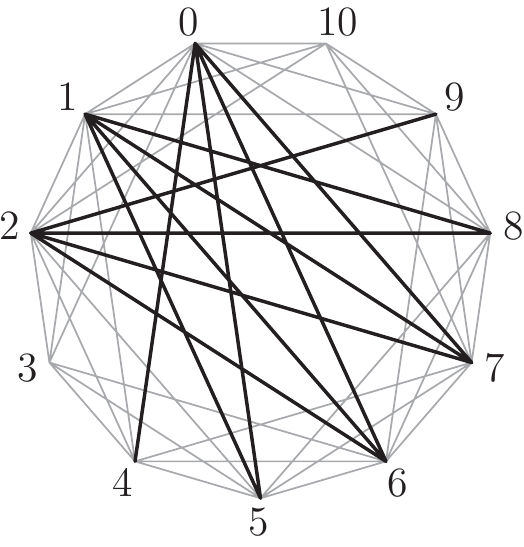}}
\caption{\small{The triangulation $T_{11,3}^{\min}$.}}\label{min}
\end{figure}

\begin{lemma}\label{connexity}
For any $k$-triangulation of the $n$-gon $T\ne T_{n,k}^{\min}$, there exists a $k$-relevant edge $f\in T\setminus T_{n,k}^{\min}$ such that the edge added by the flip of $f$ is in $T_{n,k}^{\min}$.

In particular, the $k$-triangulation $T_{n,k}^{\min}$ is the unique least element of the set of $k$-triangulations of the $n$-gon, partially ordered by $<$.
\end{lemma}

\begin{proof}
Since the second point is an immediate corollary of the first one, we only have to prove the first point.

Let $T$ be a $k$-triangulation of the $n$-gon distinct from $T_{n,k}^{min}$. Let
$$\ell=\max\{k\le i\le n-k+1\;|\;\{[0,i],[1,i+1],\ldots,[k-1,i+k-1]\}\nsubseteq T\},$$
which exists because $T\ne T_{n,k}^{min}$.

Let $0\le j\le k-1$ such that the edge $[j,\ell+j]$ is not in $T$. Let $\{[x_1,y_1],\ldots,[x_k,y_k]\}$ denote a $k$-crossing that prevent $[j,\ell+j]$ to be in $T$, with the convention that
\begin{enumerate}[(i)]
\item $x_1\cl \ldots\cl x_k\cl y_1\cl \ldots\cl y_k$,
\item if $j>0$, then $j\in\rrbracket x_j,x_{j+1}\llbracket$ and $\ell+j\in\rrbracket y_j,y_{j+1}\llbracket$,
\item if $j=0$, then $0\in\rrbracket y_k,x_1\llbracket$ and $\ell+j\in\rrbracket x_k,y_1\llbracket$.
\end{enumerate}
With this convention, we are sure that $x_k\in\llbracket k,\ell-1\rrbracket$ and $y_k\in\llbracket \ell+k,n-1\rrbracket$. If $y_k\in \rrbracket \ell+k,n-1\rrbracket$, then the set
$$\{[0,\ell],\ldots,[k-1,\ell+k],[x_k,y_k]\}$$
is a $(k+1)$-crossing of $T$. Thus $y_k=\ell+k$, and there exists an edge $[x_k,\ell+k]$ with $x_k\in\llbracket k,\ell-1\rrbracket$.

Now let $m=\min\{k\le i\le \ell-1\;|\; [i,\ell+k]\in T\}$. Let $f$ be the edge $[m,\ell+k]$, let $S$ be the $k$-star containing the angle $\angle(m,\ell+k,k-1)$ and let $R$ be the other $k$-star containing $f$. Let $s_0, \ldots,s_{k-2},s_{k-1}=k-1,s_k=m,s_{k+1},\ldots,s_{2k-1},s_{2k}=\ell+k$ denote the vertices of the $k$-star $S$ in circle order. Then $s_0\in\llbracket \ell+k+1, 0\rrbracket$, and the only way not to get a $(k+1)$-crossing is to have $s_0=0$. This implies that $s_j=j$ for all $0\le j\le k-1$. 

Let $e=e(T,f)$ be the common bisector of $R$ and $S$ and $s$ denote its vertex in $S$. Since $f=[m,\ell+k]=[s_k,s_{2k}]$ is a common edge of $R$ and $S$, it is obvious that $s\notin\{s_k,s_{2k}\}$. Moreover, since for any $0\le j\le k-1$ the interval $\rrbracket s_j,s_{j+1}\llbracket$ is empty, $s\notin\{s_{k+1},s_{k+2},\ldots,s_{2k-1}\}$. Consequently, $s\in\{s_0,\ldots,s_{k-1}\}=\llbracket 0,k-1\rrbracket$ and $e\in T_{n,k}^{\min}\setminus T$.
\end{proof}

Obviously, we get symmetric results with the $k$-triangulation $T_{n,k}^{\max}$ whose set of $k$-relevant edges is $\{[i,j]\;|\;i\in\llbracket n-k,n-1\rrbracket\;\mathrm{and}\;j\in\llbracket i+k,i-k\rrbracket\}$.

\begin{figure}
\centerline{\includegraphics[scale=1]{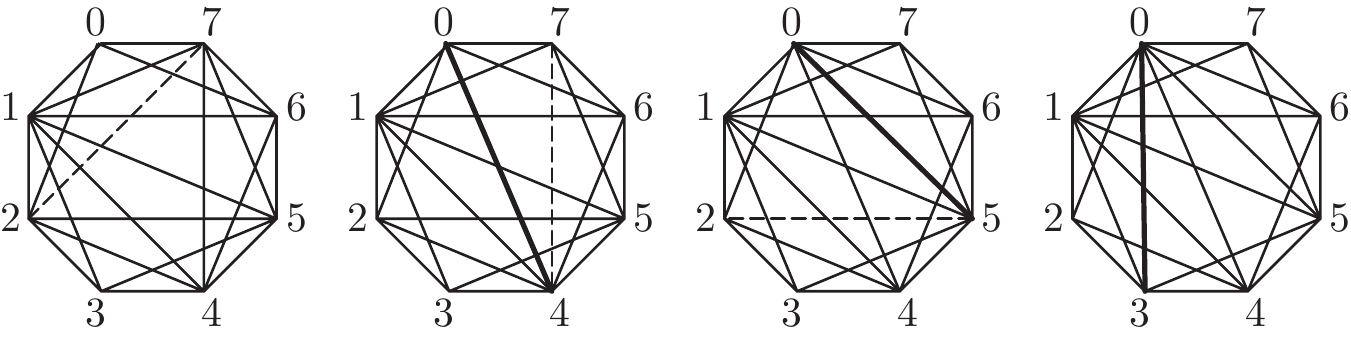}}
\caption{\small{A path of slope-decreasing flips from the $2$-triangulation of Figure~\ref{fig:2triang8points} to $T_{8,2}^{\min}$. In each image, the new edge is in bold and the dashed edge will be flipped.}}\label{fig:2triang8pointsflip}
\end{figure}

\begin{example}
\rm
Figure~\ref{fig:2triang8pointsflip} shows a path of slope-decreasing flips from the $2$-triangulation of Figure~\ref{fig:2triang8points} to $T_{8,2}^{\min}$.
\end{example}

\begin{corollary}\label{graphflips}
The graph $G_{n,k}$ is connected, regular of degree $k(n-2k-1)$, and its diameter is at most $2k(n-2k-1)$.
\end{corollary}

\begin{proof}
Let $T$ be a $k$-triangulation of the $n$-gon. The regularity follows from the fact that any of the $k(n-2k-1)$ $k$-relevant edges of $T$ can be flipped. Moreover, the previous lemma ensures that there exists a sequence of at most $k(n-2k-1)$ slope-decreasing flips from $T$ to $T_{n,k}^{\min}$. Thus any pair of $k$-triangulations is linked by a path of lenght at most $2k(n-2k-1)$ passing through $T_{n,k}^{\min}$.
\end{proof}

Obviously we would obtain the same result with any rotation of the labelling of $V_n$. In particular, for any pair $\{T,T'\}$ of $k$-triangulations, we can choose the intermediate $k$-triangulation $T''$ for our path among any rotation of $T_{n,k}^{\min}$. Consequently, there exists a path linking $T$ and $T'$ of length smaller than the average of $|T\triangle T''|+|T''\triangle T'|$ (for $T''$ among the rotations of $T_{n,k}^{\min}$). As proved in~\cite{n-gdfcp-00}, this argument improves the upper bound for 
the diameter  to be $2k(n-4k-1)$ when $n>8k^3+4k^2$.

Note that even if the improvement is asymptotically not relevant, for the case $k=1$, the improved bound of $2n-10$ is actually the exact diameter of the associahedron for large values of $n$~\cite{stt-rdthg-88}. For $k>1$, we only have the following lower bound:

\begin{lemma}\label{diamin}
If $n\ge 4k$, then the diameter of $G_{n,k}$ is at least $k(n-2k-1)$.
\end{lemma}

\begin{proof}
We only have to find two $k$-triangulations of the $n$-gon with no $k$-relevant edges in common. Let $Z$ be the following subset of $k$-relevant edges of $E_n$:
\begin{eqnarray*}
Z & = & \left\{[q-1,-q-k]\;|\; 1\le q\le \lfloor n/2-k\rfloor\right\}\\
&&\cup\left\{[q,-q-k]\;|\; 1\le q\le \lfloor (n-1)/2-k\rfloor\right\}.
\end{eqnarray*}
We say that $Z$ is a \emph{$k$-zigzag} of $E_n$ (see Fig.~\ref{zigzags}). Let $\rho$ be the rotation $t\mapsto t+1$. Observe that since $n\ge 4k$, the $2k$ $k$-zigzags $Z,\rho(Z),\ldots,\rho^{2k-1}(Z)$ are disjoint. Moreover,
\begin{enumerate}[(i)]
\item there is no $2$-crossing in a $k$-zigzag, so that there is no $(k+1)$-crossing in the union of $k$ of them;
\item a $k$-zigzag contains $n-2k-1$ $k$-relevant edges, so that the union of $k$ disjoint $k$-zigzags contains $k(n-2k-1)$ $k$-relevant edges.
\end{enumerate}
According to Corollary~\ref{starsenumeration}, this proves that the union of $k$ disjoint $k$-zigzag is the set of $k$-relevant edges of a $k$-triangulation. Thus, we obtain the two $k$-triangulations we were looking for with the sets of $k$-relevant edges $\bigcup_{i=0}^{k-1} \rho^i(Z)$ and $\bigcup_{i=k}^{2k-1} \rho^i(Z)$.
\begin{figure}
\centerline{\includegraphics[scale=1]{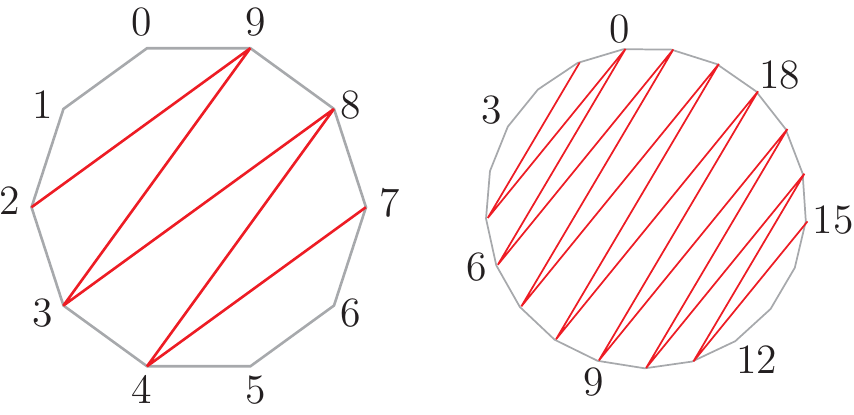}}
\caption{\small{Examples of $k$-zigzags of $E_n$ for $(n,k)=(10,2)$ and $(21,3)$.}}\label{zigzags}
\end{figure}
\end{proof}

We want to observe that Lemma~\ref{connexity} is the way how~\cite{n-gdfcp-00} and \cite{dkm-lahp-02} prove that the number of edges in all $k$-triangulations of the $n$-gon is the same. Compare it to our direct proof in Corollary~\ref{starsenumeration}.


\section{$k$-ears and $k$-colorable $k$-triangulations}\label{sectionears}

Let us assume that $n\ge 2k+3$. We call a \emph{$k$-ear} any edge of length $k+1$ in a $k$-triangulation. We say that a $k$-star is \emph{external} if it contains at least one $k$-boundary edge (and \emph{internal} otherwise).
It is well known and easy to prove that the number of ears in any triangulation equals its number of internal triangles plus $2$. In this section, we prove that the number of $k$-ears in any $k$-triangulation equals its number of internal $k$-stars plus $2k$. Then we characterize the triangulations that have no internal $k$-star in terms of colorability of their intersection complex.

Let $S$ be a $k$-star and $[u,v]$ be an edge of $S$ such that $S$ lies on the positive side of the oriented edge from $u$ to $v$. We say that $[u,v]$ is a \emph{positive ear} of $S$ if $|\llbracket v,u\llbracket|=k+1$. Said differently, $[u,v]$ is an ear, and $S$ is the unique (by Corollary~\ref{cor:numberofstars}) $k$-star on the ``outer" side of it.

\begin{lemma}\label{starwithboundary}
Let $b\ge 1$ be the number of $k$-boundary edges of an external $k$-star $S$. Then, $S$ has exactly $b-1$ positive ears. Moreover, $k$-boundary edges and positive ears of $S$ form an alternating path in the $k$-star.
\end{lemma}

\begin{proof}
Observe first that if $[x,x+k]$ is a $k$-boundary edge of $S$, then all $k+1$ vertices of $\llbracket x,x+k\rrbracket$ are vertices of $S$. Since $S$ has $2k+1$ vertices, this implies that if $[x,x+k]$ and $[y,y+k]$ are two $k$-boundary edges of $S$, then $\llbracket x,x+k\rrbracket$ and $\llbracket y,y+k\rrbracket$ intersect, that is for example $y\in \llbracket x,x+k\rrbracket$. But then, all edges $[i,i+k]$ ($x\cle i\cle y$) are $k$-boundary edges of $S$, and all edges $[i,i+k+1]$ ($x\cle i\cl y$) are positive ears of $S$. Thus, $k$-boundary edges and positive ears of $S$ form an alternating path in the $k$-star, begining and ending by a $k$-boundary edge. In particular, $S$ has exactly $b-1$ positive ears.
\end{proof}

\begin{corollary}\label{earsenumeration}
The number of $k$-ears in a $k$-triangulation $T$ equals the number of internal $k$-stars plus $2k$.
In particular, $T$ contains at least $2k$ $k$-ears.
\end{corollary}

\begin{proof}
For any $0\le i\le 2k+1$, let $\mu_i$ denote the number of $k$-stars of $T$ with exactly $i$ $k$-boundary edges. Let $\nu$ denote the number of $k$-ears of $T$. Then
$$\sum_{i=0}^{2k+1} \mu_i=n-2k \quad \mathrm{and}\quad \sum_{i=1}^{2k+1} i\mu_i=n.$$
Since any $k$-ear is a positive ear of exactly one $k$-star, the previous lemma ensures moreover that
$$\sum_{i=1}^{2k+1} (i-1)\mu_i=\nu.$$
Thus, we obtain $\nu=n-(n-2k)+\mu_0=2k+\mu_0$.
\end{proof}

\begin{example}
\rm
The $2$-triangulation in Figure~\ref{fig:2triang8points} has five $2$-ears and one internal $2$-star (the $2$-star (a) in Figure~\ref{fig:2triang8pointsstars}).
\end{example}

We are now interested in a characterization of the $k$-triangulations that have exactly $2k$ $k$-ears, or equivalently that have no internal $k$-star. We need the following definitions.

We say that a $k$-triangulation is \emph{$k$-colorable} if it is possible to color its $k$-relevant edges with $k$ colors such that there is no monochromatic $2$-crossing. Observe that, if this happens, then every $k$-crossing uses an edge of each color.

\begin{definition}
A \emph{$k$-accordion} of $E_n$ is a set $Z=\{[a_i,b_i]\;|\; 1\le i\le n-2k-1\}$ of $n-2k-1$ edges such that $b_1=a_1+k+1$ and for any $2\le i\le n-2k-1$, the edge $[a_i,b_i]$ is either $[a_{i-1},b_{i-1}+1]$ or $[a_{i-1}-1,b_{i-1}]$ (see Fig.~\ref{accordion}).
\end{definition}

\begin{figure}
\centerline{\includegraphics[scale=1]{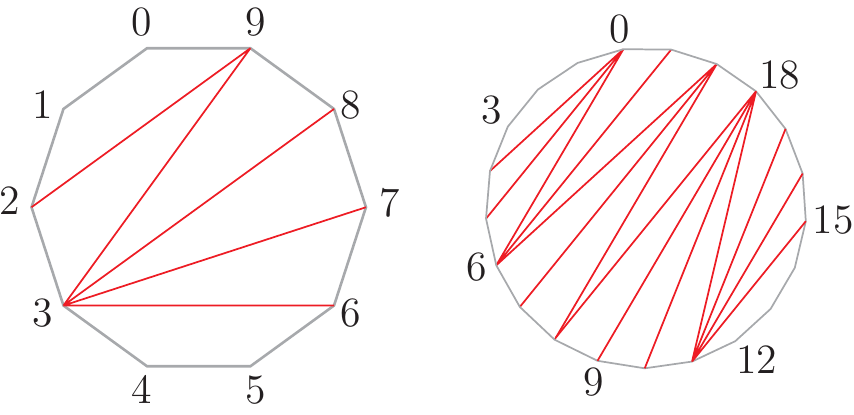}}
\caption{\small{Examples of $k$-accordions of $E_n$ for $(n,k)=(10,2)$ and $(21,3)$.}}\label{accordion}
\end{figure}

Observe that, in this definition, $b_i=a_i+k+i$ for any $1\le i\le n-2k-1$. Moreover, it is easy to check that the definition is equivalent to being a set of $n-2k-1$ $k$-relevant edges of $E_n$ without $2$-crossing.

\begin{lemma}
The union of $k$ disjoint $k$-accordions of $E_n$ is the set of $k$-relevant edges of a $k$-triangulation of the $n$-gon.
\end{lemma}

\begin{proof}
Observe that
\begin{enumerate}[(i)]
\item there is no $2$-crossing in a $k$-accordion, so that there is no $(k+1)$-crossing in the union of $k$ of them;
\item a $k$-accordion contains $n-2k-1$ $k$-relevant edges, so that the union of $k$ disjoint $k$-accordions contains $k(n-2k-1)$ $k$-relevant edges.
\end{enumerate}
This lemma thus follows from Corollary~\ref{starsenumeration}.
\end{proof}

We have already met some particular $k$-accordions both when we constructed the triangulation $T_{n,k}^{\min}$ and in the proof of Lemma~\ref{diamin}. The two types of $k$-accordions in these examples (``fan" and ``zigzag") are somehow the two extremal examples of them.

\begin{theorem}\label{colorable}
Let $T$ be a $k$-triangulation, with $k>1$. The following properties are equivalent for $T$:
\begin{enumerate}[(i)]
\item $T$ is $k$-colorable;
\item there exists a $k$-coloring of the $k$-relevant edges of $T$ such that no $k$-star of $T$ contains three edges of the same color.
\item $T$ has no internal $k$-star;
\item $T$ contains exactly $2k$ $k$-ears;
\item $T$ contains exactly $2k$ edges of each length $k+1,\dots, \lfloor (n-1)/2 \rfloor$ (and $k$ of length $n/2$ if $n$ is even);
\item the set of $k$-relevant edges of $T$ is the union of $k$ disjoint $k$-accordions.
\end{enumerate}
\end{theorem}

Observe that for $k=1$, only (ii), (iii), (iv), (v) and (vi) are equivalent, while (i) holds for every triangulation.

\begin{proof}[Proof of Theorem~\ref{colorable}]
When $k>1$, any three edges of a $k$-star form at least one $2$-crossing. Thus, any $k$-coloring of $T$ without monochromatic $2$-crossing is such that no $k$-star of $T$ contains three edges of the same color, and (i) implies (ii).

Let $S$ be a $k$-star of a $k$-triangulation whose $k$-relevant edges are colored with $k$ colors. If all edges of $S$ are $k$-relevant, then, by the pigeon-hole principle, there is a color that colors three edges of $S$. Thus (ii) implies (iii).
Corollary~\ref{earsenumeration} ensures that (iii) and (iv) are equivalent.

Since (vi)$\Rightarrow$(v)$\Rightarrow$(iv) and (vi)$\Rightarrow$(i) are trivial, it only remains to prove that (iv) implies (vi). 

For this, we give an algorithm that finds the $k$ disjoint $k$-accordions in a $k$-triangulation $T$ with $2k$ $k$-ears.  Recall that if $S$ is an external $k$-star with $b$ $k$-boundary edges, then $S$ has $b-1$ positive ears, and that $k$-boundary edges and positive ears of $S$ form an alternating path in the $k$-star. Thus, the edges of $S$ which are neither $k$-boundary nor positive ears of $S$ form a path of even length. This defines naturally a ``pairing" of them: we say that the first and the second (resp. the third and the fourth, {\it etc.}) edges of this path form a \emph{pair} of edges in $S$.

Consider now a $k$-ear $e_1$ of $T$. Let $S_0$ be the $k$-star of $T$ for which $e_1$ is a positive ear. Let $S_1$ be the other $k$-star of $T$ containing $e_1$. Let $e_2$ be the pair of $e_1$ in $S_1$. Let $S_2$ be the other $k$-star of $T$ containing $e_2$ and let $e_3$ be the pair of $e_2$ in $S_2$. Let continue so untill we reach a $k$-ear. It is obvious that we get a $k$-accordion of $E_n$. To get another $k$-accordion, we do the same with a $k$-ear which is neither $e_1$ nor $e_{n-2k-1}$. 

To prove the correctness of this algorithm, we only have to prove that the $k$-accordions we construct are disjoint. Suppose for example that two of them $\{e_1,\ldots,e_{n-2k-1}\}$ and $\{f_1,\ldots,f_{n-2k-1}\}$ intersect. Let $i,j$ be such that $e_i=f_j$. Let $S$ be a $k$-star containing $e_i$. Then by construction, either $e_{i+1}$ or $e_{i-1}$ (resp. either $f_{j+1}$ or $f_{j-1}$) is the pair of $e_i$ in $S$. Thus, either $e_{i+1}=f_{j+1}$, or $e_{i+1}=f_{j-1}$. By propagation, we get that $\{e_1,\ldots,e_{n-2k-1}\}=\{f_1,\ldots,f_{n-2k-1}\}$
\end{proof}

Observe that the above proof of (iv)$\Rightarrow$(vi) also gives uniqueness (up to permutation of colors) of the $k$-coloring (decomposition into accordions)
of a $k$-colorable $k$-triangulaiton $T$. Indeed, any $k$-coloring has to respect the pairing of edges in all $k$-stars of $T$.

Let us also make the remark that part (v) implies that every $k$-colorable $k$-triangulation contains exactly  $k(n-2p-1)$ $p$-relevant edges, for any $k\le p\le \lfloor\frac{n-1}{2}\rfloor$. It is proved by a flip method in~\cite{n-gdfcp-00} that any $k$-triangulation of the $n$-gon contains at most this same number of $p$-relevant edges. We will prove this result in the next section, as an application of the flattenings of $k$-stars in $k$-triangulations.

Let us conclude this section with the following remark. An easy ``intuitive model" for what a $k$-triangulation is is just a superposition of $k$ triangulations. Even if this model is sometimes usefull, the results in this section say that it is also misleading:
\begin{itemize}
\item Theorem~\ref{colorable} says that the structure of $k$-triangulations obtained in this way is really particular: the number of edges of each length is $2k$, all $k$-stars are external, {\it etc.}
\item The number of $k$-accordions of $E_n$ containing a given $k$-ear of $E_n$ is $2^{n-2k-2}$. In particular, the number of $k$-colorable $k$-triangulations of $E_n$ is at most ${n \choose k}2^{k(n-2k-2)}\le 2^{(k+1)n}$. This is much smaller than the total number of $k$-triangulations, which for constant $k$ equals $4^{nk}$ modulo a polynomial factor in $n$ (see Theorem~\ref{triangenumeration}).

\item Let $T$ be any triangulation with only two ears, and let $Z$ denote the set of its $2$-relevant edges. Clearly, $Z$ is a $2$-accordion, and it is easy to see that there exists a $2$-accordion $Z'$ disjoint from it, so that $Z'$ complete $Z$ to give the set of $2$-relevant edges of a $2$-colorable $2$-triangulation. Surprisingly, this simple property fails when $k\ge3$: even if a $k$-triangulation $T$ is $k$-colorable, it is not always possible to find a $(k+1)$-accordion disjoint from $T$ (see Fig.~\ref{ctrexm}).

\begin{figure}
\centerline{\includegraphics[scale=1]{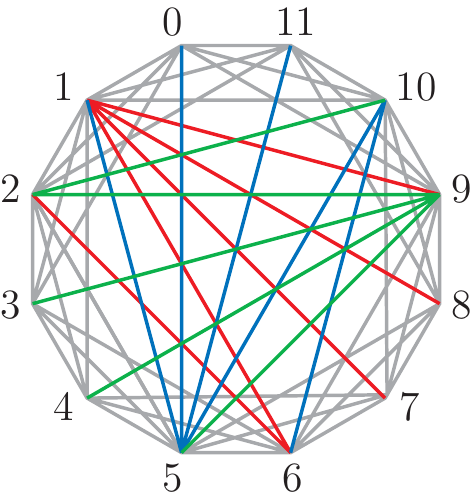}}
\caption{\small{A $3$-colorable $3$-triangulation of the $12$-gon such that there is no $4$-accordion disjoint from it.}}\label{ctrexm}
\end{figure}
\end{itemize}


\section{Flattening a $k$-star, inflating a $k$-crossing}\label{sectionflatinflat}

The goal of this section is to describe in terms of $k$-stars an operation that connects $k$-triangulations of $n$ and of $n+1$ vertices. This operation, already present in~\cite{j-gt-03}, in~\cite{n-gdfcp-00}, and when $k=2$ in~\cite{e-btdp-06}, is useful for recursive arguments and it was a step in all previous proofs of the flipability of $k$-relevant edges (Corollary~\ref{graphflips}).

\begin{definition}
Let $T$ be a $k$-triangulation of the $(n+1)$-gon, and $e=[s_0,s_0+k]$ be a $k$-boundary edge of $T$.
Let $s_0,s_1=s_0+1,\ldots,s_k=s_0+k,s_{k+1},\ldots,s_{2k}$ be the vertices of the unique $k$-star $S$ of $T$ containing $e$, in their circle order.

We call \emph{flattening} of $e$ in $T$ the set of edges $\underline{T}_e$ whose underlying set of points is the set $V_{n+1}\smallsetminus\{s_0\}$ and which is contructed from $T$ as follows (see Fig.~\ref{flatinfl}):
\begin{enumerate}[(i)]
\item for any edge of $T$ whose vertices are not in $\{s_0,\ldots,s_k\}$ just copy the edge;
\item forget all the edges $[s_0,s_i]$, for $1\le i\le k$;
\item replace any edge of the form $[s_i,t]$ with $0\le i\le k$, and $s_k\cl t \cle s_{k+i}$ (resp. $s_{k+i+1}\cle t\cl s_0$) by the edge $[s_i,t]$ (resp. $[s_{i+1},t]$).
\end{enumerate}
\end{definition}

\begin{figure}
\centerline{\includegraphics[scale=1]{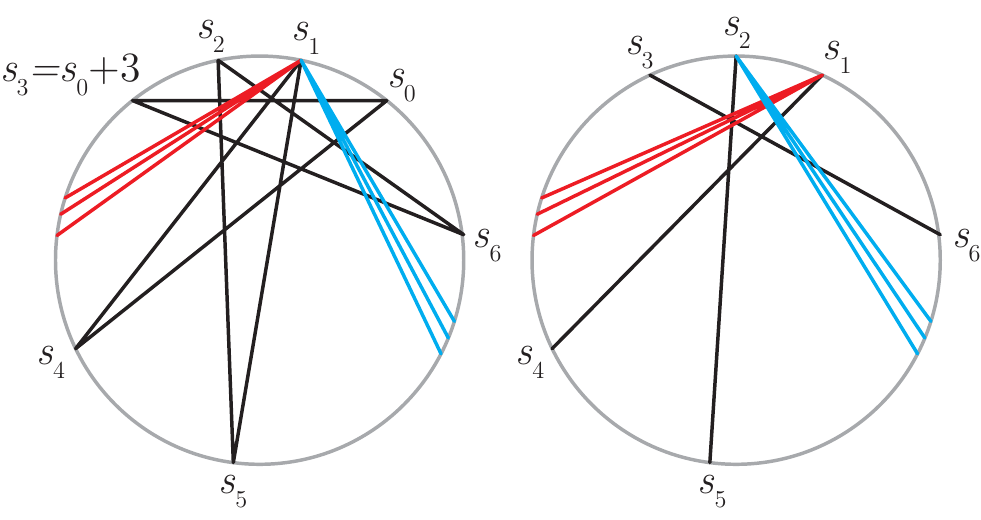}}
\caption{\small{Flattening a $3$-boundary edge - inflating a $3$-crossing.}}\label{flatinfl}
\end{figure}

\begin{remark}
\rm
We sometimes call this operation ``flattening of the $k$-star $S$". This is a somehow more graphical description of the operation, but it is also a slight abuse of language. Indeed, if $S$ has more than one $k$-boundary edge, the result of the flattening of $S$ depends on which $k$-boundary edge we flatten.
\end{remark}

We want to prove that $\underline{T}_e$ is a $k$-triangulation of the $n$-gon. Observe first that
\begin{enumerate}
\item If $f$ is a $k$-relevant edge of $\underline{T}_e$, then
\begin{enumerate}[(i)]
\item either $f$ is of the form $[s_i,s_{k+i}]$, for some $1\le i\le k$, and then it arises as the glueing of two edges $f'=[s_{i-1},s_{k+i}]$ and $f''=[s_i,s_{k+i}]$ of the initial triangulation $T$;
\item or $f$ is not of the previous form, and then it arises from a unique $k$-relevant edge $f'$ of $T$.
\end{enumerate}

\item If $f$ and $g$ are two $k$-relevant edges of $\underline{T}_e$ arising from $f'$ and $g'$ respectively, then
\begin{enumerate}[(i)]
\item either $f$ and $g$ do not cross, and then $f'$ and $g'$ do not cross; 
\item or $f$ and $g$ do cross, and then $f'$ and $g'$ do cross, unless the following happens: there exists $i$ in $\{1,\ldots,k\}$, and $u,v$ two vertices such that $s_k\cl u\cle s_{i+k}\cl s_{i+k+1}\cle v\cl s_0$ and $f=[s_i,u]$, $g=[s_{i+1},v]$, $f'=[s_i,u]$ and $g'=[s_i,v]$. Such a configuration is said to be a \emph{hidden configuration} (see Fig.~\ref{flatinfl}).
\end{enumerate}
\end{enumerate}

It is easy to derive from this that $|\underline{T}_e|=|T|-2k=k(2n-2k-1)$ and that any subset of $E_n$ that properly contains $\underline{T}_e$ contains a $(k+1)$-crossing. However, this is not sufficient to conclude that $\underline{T}_e$ is a $k$-triangulation (see Fig.~\ref{ctrexm2} for a counter-example). Thus we have to prove that $\underline{T}_e$ is $(k+1)$-crossing-free. Note that this provides a third proof (which this time is recursive on $n$) of the number of edges of a $k$-triangulation of the $n$-gon.

\begin{figure}
\centerline{\includegraphics[scale=1]{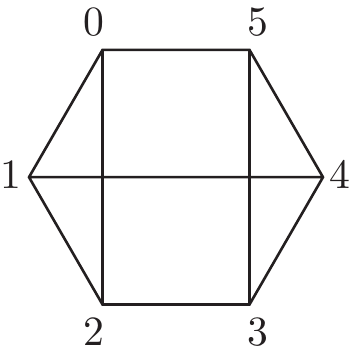}}
\caption{\small{A subset of $9$ edges of $E_6$ such that any proper superset contains a $2$-crossing but which is not a triangulation.}}\label{ctrexm2}
\end{figure}

\begin{lemma}
The set $\underline{T}_e$ is $(k+1)$-crossing-free. Hence, it is a $k$-triangulation of the $n$-gon.
\end{lemma}

\begin{proof}
Suppose that $\underline{T}_e$ contains a $(k+1)$-crossing $F$. Let denote $f_0=[x_0,y_0],\ldots,f_k=[x_k,y_k]$ the edges of $F$ ordered such that $x_0\cl x_1\cl\ldots\cl x_k\cl y_0\cl y_1\cl \ldots\cl y_k$. Let $f'_0=[x'_0,y'_0],\ldots,f'_k=[x'_k,y'_k]$ be $k+1$ edges of $T$ that give (when we flatten $e$) the edges $f_0,\ldots,f_k$ respectively.

It is clear that if there exists no $0\le i\le k$ such that the four edges $(f_i,f_{i+1},f'_i,f'_{i+1})$ form a hidden configuration, then the edges $f'_0,\ldots,f'_k$ form a $(k+1)$-crossing of $T$, which is impossible. Thus we can suppose that the number of hidden configurations in the set $\{(f_i,f_{i+1},f'_i,f'_{i+1})\;|\; 0\le i\le k\}$ is at least $1$. We can also assume that this number is minimal, that is that we can not find a $(k+1)$-crossing $G$ of $\underline{T}_e$ arising from a set $G'$ of edges of $T$ such that the number of hidden configurations in the set $\{(g_i,g_{i+1},g'_i,g'_{i+1})\;|\; 0\le i\le k\}$ is strictly less than in $\{(f_i,f_{i+1},f'_i,f'_{i+1})\;|\; 0\le i\le k\}$. Here, we raise an absurdity by finding such sets $G$ and $G'$.

Let $i$ in $\{0,\ldots,k\}$ be such that $(f_i,f_{i+1},f'_i,f'_{i+1})$ is a hidden configuration. We can suppose that $x_i=s_i$ and $x_{i+1}=s_{i+1}$ (if this is not the case, we renumber the edges of $F$ such that this be true). Thus we know that $y_i\cle s_{i+k}\cl s_{i+k+1}\cle y_{i+1}$. Let $p$ denote the first integer before $i$ such that $x_p \ne s_p$ or $s_{p+k}\cl y_p$, and $q$ denote the first integer after $i+1$ such that $x_q \ne s_q$ or $y_q\cl s_{q+k}$.

Let $G$ be the set of $k+1$ edges of $\underline{T}_e$ deduced from $F$ as follows:
\begin{itemize}
\item for all $p<i<q$, let $g_i$ be $[s_i,s_{i+k}]$,
\item for all $q\le i\le p$, let $g_i$ be $f_i$.
\end{itemize}
Let $G'$ be the set of $k+1$ edges of $T$ constructed as follows:
\begin{itemize}
\item for all $p<i<q$, let $g'_i$ be $[s_i,s_{i+k}]$,
\item for all $q\le i\le p$, let $g'_i$ be $f'_i$.
\end{itemize}
		
It is quite clear that $G$ is a $(k+1)$-crossing of $\underline{T}_e$ arising from $G'$. We just have to verify that the number of hidden configurations in the set $\{(g_i,g_{i+1},g'_i,g'_{i+1})\;|\; 0\le i\le k\}$ is less than in $\{(f_i,f_{i+1},f'_i,f'_{i+1})\;|\; 0\le i\le k\}$. But
\begin{enumerate}[(i)]
\item the number of hidden configurations in $\{(g_i,g_{i+1},g'_i,g'_{i+1})\;|\; q\le i<p\}$ is exactly the same as in $\{(f_i,f_{i+1},f'_i,f'_{i+1})\;|\; q\le i<p\}$,
\item there is no hidden configuration in $\{(g_i,g_{i+1},g'_i,g'_{i+1})\;|\; p<i<q-1\}$, whereas there is one in $\{(f_i,f_{i+1},f'_i,f'_{i+1})\;|\; p<i<q-1\}$,
\item the edges $(g_p,g_{p+1},g'_p,g'_{p+1})$ (resp. $(g_{q-1},g_q,g'_{q-1},g'_q)$) do not form a hidden configuration.
\end{enumerate}
\end{proof}

\medskip
We now define the inverse transformation.

\begin{definition}
Let $T$ be a $k$-triangulation of the $n$-gon and $E$ be a $k$-crossing. Let $[s_1,s_{1+k}],\ldots,[s_k,s_{2k}]$ denote its edges, with $s_1\cl s_2 \cl \ldots \cl s_{2k}$. Assume further that $s_1,\ldots,s_k$ are consecutive in the cyclic order of the $n$-gon (that is $s_k=s_1+k-1$). We call $E$ an \emph{external} $k$-crossing.
Let $s_0$ be a new vertex on the circle, between $s_1-1$ and $s_1$.

We call \emph{inflating} of $E$ at $\llbracket s_1,s_k\rrbracket$ in $T$ the set of edges $\overline{T}^E$ whose underlying set of points is the set $V_n\cup\{s_0\}$ and which is contructed from $T$ as follows (see Fig.~\ref{flatinfl}):
\begin{enumerate}[(i)]
\item for any edge of $T$ whose vertices are not in $\{s_1,\ldots,s_k\}$ just copy the edge;
\item add all the edges $[s_0,s_i]$, for $1\le i\le k$;
\item replace any edge of the form $[s_i,t]$ with $1\le i\le k$, and $s_k\cl t \cle s_{k+i}$ (resp. $s_{k+i}\cle t\cl s_1$) by the edge $[s_i,t]$ (resp. $[s_{i-1},t]$).
\end{enumerate}
\end{definition}

\begin{remark}
\rm
We are abusing notation: if $E$ contains more than $k$ consecutive vertices, then the result of the inflating of $E$ depends on the $k$-consecutive points we choose. It would also be an abuse of language to say that we inflate the cyclic interval $\llbracket s_1,s_k\rrbracket$ since several $k$-crossings may be adjacent to these $k$ points. For example, when $k=1$, we have to specify both an edge and a vertex of this edge to define an inflating.
\end{remark}

Observe that:
\begin{enumerate}
\item Any $k$-relevant edge $f$ of $\overline{T}^E$ arises from a unique edge $f'$ of $T$.

\item If $f$ and $g$ are two $k$-relevant edges of $\overline{T}^E$ that cross, and $f'$ and $g'$ are the two edges of $T$ that give $f$ and $g$ respectively, then $f'$ and $g'$ cross as well.
\end{enumerate}

Point~(2) ensures that $\overline{T}^E$ is $(k+1)$-crossing-free. Moreover, it is easy to see that $|\overline{T}^E|=|T|+2k=k(2(n+1)-2k-1)$. Thus, by Corollary~\ref{starsenumeration}, we get:

\begin{lemma}
$\overline{T}^E$ is a $k$-triangulation of the $(n+1)$-gon.
\end{lemma}

\begin{theorem}
Flattening and inflating are inverse operations. More precisely:
\begin{enumerate}[(i)]
\item if $e$ is a $k$-boundary edge of a $k$-triangulation $T$, and $E$ denotes the $k$-crossing of $\underline{T}_e$ consisting of edges that arise as glueing of two edges of $T$, then $\overline{\underline{T}_e}^E=T$;
\item if $E$ is a $k$-crossing with $k$ consecutive vertices $s_1,\ldots,s_k$, and $e$ denotes the edge $[s_0,s_k]$ of $\overline{T}^E$, then $\underline{\overline{T}^E}_e=T$.
\end{enumerate}
\end{theorem}

\begin{example}
\rm
Figure~\ref{fig:2triang8pointsinflat} shows the inflating of $2$-crossing $\{[1,4],[2,7]\}$ at $\llbracket 1,2\rrbracket$ in the $2$-triangulation of Figure \ref{fig:2triang8points}. The new $2$-star is colored in red.
\end{example}

\begin{figure}
\centerline{\includegraphics[scale=1]{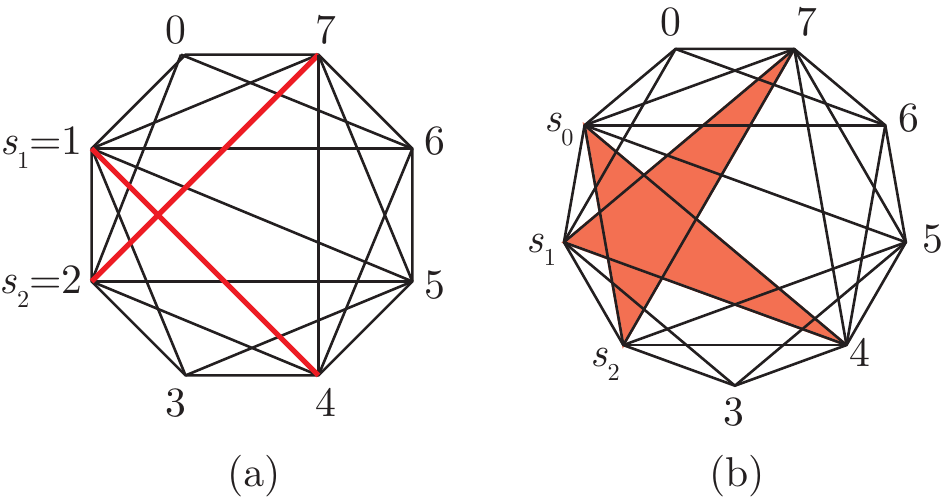}}
\caption{\small{The inflating of the $2$-crossing $\{[1,4],[2,7]\}$ at $\llbracket 1,2\rrbracket$ in the $2$-triangulation of Figure \ref{fig:2triang8points}.}}\label{fig:2triang8pointsinflat}
\end{figure}

We can now use these inverse tranformations to prove additional results on $k$-triangulations by induction. 

\begin{corollary}\label{enumerationinduction}
Let $N(n,k)$ denote the number of $k$-triangulations of an $n$-gon. The quotient $N(n+1,k)/N(n,k)$ equals the average number of $k$-crossings on the first $k$ points among all $k$-triangulations of the $n$-gon.
\end{corollary}

For example,
\begin{enumerate}[(i)]
\item For $k=1$, we get that $N(n+1,1)/N(n,1)$ equals the average degree of vertex 1 in triangulations of the $n$-gon. Since the average degree of all triangulations of the $n$-gon is the same and equal to $(4n-6)/n$ we recover the well-known recursion for Catalan numbers:
$$C_{n-1}=\frac{4n-6}{n}C_{n-2}.$$

\item For $n=2k+1$ we have that $N(2k+1,k)=1$ (the unique $k$-triangulation is the complete graph)
and the number of $k$-crossings using the first $k$ vertices in this $k$-triangulation is $k+1$ (any choice of $k$ of the last $k+1$ vertices gives one $k$-crossing). In particular, we recover the fact that $N(2k+2,k)=k+1$ (see Example~\ref{exm:2k+2}).

\item Unfortunately, for $n>2k+1$ it is not true that the number of $k$-crossings using $k$ consecutive vertices is independent of the $k$-triangulation. Otherwise we would have that $N(n+1,k)\cdot n /N(n,k)$ is an integer, equal to that number (as happens in the case of triangulations).
\end{enumerate}

\medskip
The following lemma is another example of the use of a recursive argument. It is equivalent to Theorem 10 in~\cite{n-gdfcp-00}, where the proof uses flips.

\begin{lemma}
Any $k$-triangulation of the $n$-gon contains at most $k(n-2p-1)$ $p$-relevant edges, for any $k\le p\le \frac{n-1}{2}$.
\end{lemma}

\begin{proof}
When $n=2p+1$, it is true since there are no $p$-relevant edges.

Suppose now that $n>2p+1$. Let $e=[u,u+k]$ be a $k$-boundary edge of $T$. It is easy to check that if $g$ is a $k$-relevant edge of $T$ of length $\ell$, then the corresponding edge $g'$ in $\underline{T}_e$ has length $\ell$ or $\ell-1$, and the later is possible only if $g=[v,v+\ell]$ with $v\cle u\cl u+k\cle v+\ell$.

Let $E$ be the $k$-crossing of $\underline{T}_e$ that arise from the edges of the $k$-star of $T$ containing $e$. Let $F$ be the set of edges of $E$ of length at least $p$. Let $G$ denote the set of edges of length $p$ of $\underline{T}_e\setminus E$ arising from an edge of length $p+1$ of $T$.

Any non $p$-relevant edge of $\underline{T}_e\smallsetminus (E\cup G)$ (resp. of $E\setminus F$) arise from one (resp. two) non $p$-relevant edge of $T$. Thus we obtain by induction that the number of $p$-relevant edges is at most $k(n-1-2p-1)+|F|+|G|$.

To conclude, we only have to observe that any two edges of $F\cup G$ are crossing, which implies that $|F|+|G|=|F\cup G|\le k$.
\end{proof}




To close this section, observe that the result of flattening several $k$-boundary edges of a $k$-triangulation is independant of the order.
Indeed, let $e$ and $f$ be two distinct $k$-boundary edges of a $k$-triangulation $T$. Let $e'$ (resp. $f'$) denote the edge of $\underline{T}_f$ (resp. $\underline{T}_e$) arising from $e$ (resp. $f$). Then $e'$ (resp. $f'$) is a $k$-boundary edge of $\underline{T}_f$ (resp. $\underline{T}_e$) and
$$\underline{\underline{T}_f}_{e'}=\underline{\underline{T}_e}_{f'}.$$
In particular, one can define the flattening of a set of $k$-boundary edges. 

Similarly, it is possible to define the inflating of a set of edges-disjoint $k$-crossings of a $k$-triangulation.


\section{Further topics and open questions}\label{sectionopen}

In this section, we discuss further topics and open questions related to multi-triangulations. We hope that at least some of them may be easier to prove and understand looking at $k$-triangulations as ``complexes of $k$-stars".

\bigskip

\paragraph{{\sc The number of $k$-triangulations}}

Remember that Jonsonn proved in~\cite{j-gt-03} that

\begin{theorem}\label{triangenumeration}
The number of $k$-triangulations of the $n$-gon is equal to the determinant
$$\det(C_{n-i-j})_{1\le i,j\le k}=\left|\begin{pmatrix} C_{n-2} & C_{n-3} & \ldots & C_{n-k} & C_{n-k-1} \\ C_{n-3} & C_{n-4} & \ldots & C_{n-k-1} & C_{n-k-2} \\ \vdots & \vdots & \ddots & \vdots & \vdots \\ C_{n-k-1} & C_{n-k-2} & \ldots & C_{n-2k+1} & C_{n-2k} \end{pmatrix}\right|,$$
where $C_m=\frac{1}{m+1}{ 2m \choose m}$ denotes the $m$-th \emph{Catalan number}.
\end{theorem}

On the other hand, the Catalan determinant in this statement also counts two other classes of objects:
\begin{enumerate}[(i)]
\item Dyck $k$-paths of semi-length $n-2k$:

A \emph{Dyck path} of \emph{semi-length} $\ell$ is a lattice path using north steps $N=(0,1)$ and east steps $E=(1,0)$ starting from $(0,0)$ and ending at $(\ell,\ell)$, and such that it never goes below the diagonal $y=x$ (see Fig.~\ref{dyck}~(a)). We call \emph{Dyck $k$-path} of \emph{semi-length} $\ell$ any $k$-tuple $(d_1,\ldots,d_k)$ of Dyck paths of semi-length $\ell$ such that each $d_i$ never goes above $d_{i-1}$, for $2\le i\le k$ (see Fig.~\ref{dyck}~(b)).

\begin{figure}
\centerline{\includegraphics[scale=1]{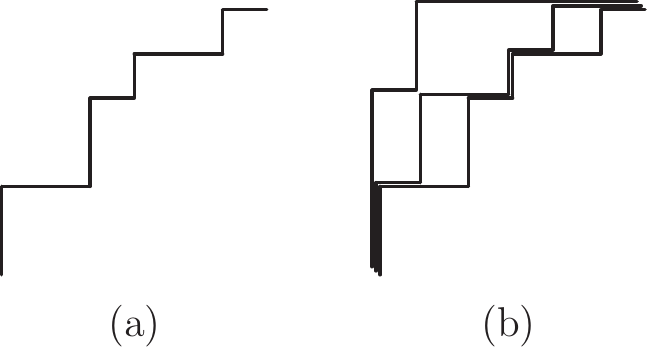}}
\caption{\small{A Dick path (a) and a Dyck $3$-path (b) of semi-length $6$.}}\label{dyck}
\end{figure}

That the Catalan determinant of Theorem~\ref{triangenumeration} also counts the number of Dyck $k$-paths of semi-length $n-2k$ is proved in~\cite{dcv-ecytbh-86}.

\medskip

\item Maximal $k$-diagonal-free fillings of the triangular polyomino $A_n$:

A \emph{stack polyomino} of shape $(s_1,\ldots,s_r)$ (where there exists $t\in\{1,\ldots,r\}$ such that $0<s_1 \le \ldots\le s_t$ and $s_t\ge \ldots\ge s_r>0$) is the following subset of $\mathbb{Z}^2$, whose elements are called \emph{boxes}:
$$\{(i,j)\;|\; 1\le i\le r, 0\le j\le s_i\}.$$
A \emph{filling} of a polyomino is an assignment of $0$ and $1$ to its boxes.
For $\ell\in\mathbb{N}$, an \emph{$\ell$-diagonal} of a filling of a polyomino $\Lambda$ is a chain $(\alpha_1,\beta_1),\ldots,(\alpha_\ell,\beta_\ell)$ of $\ell$ boxes of $\Lambda$ filled with $1$ and such that
\begin{itemize}
\item $\alpha_1<\alpha_2<\ldots<\alpha_\ell$ and $\beta_1<\beta_2<\ldots<\beta_\ell$, and
\item the rectangle $\{(\alpha,\beta)\;|\; \alpha_1\le \alpha\le\alpha_\ell, \beta_1\le \beta\le\beta_\ell\}$ is a subset of $\Lambda$.
\end{itemize}
See Figure~\ref{polyominoes1} for an illustration (we use the matrix convention to index the boxes of a polyomino; that is, the row index increases as we go down). See~\cite{j-gtdfssp-05, k-gdidcffs-06, r-idsfmp-07} for more details.

\begin{figure}
\centerline{\includegraphics[scale=1]{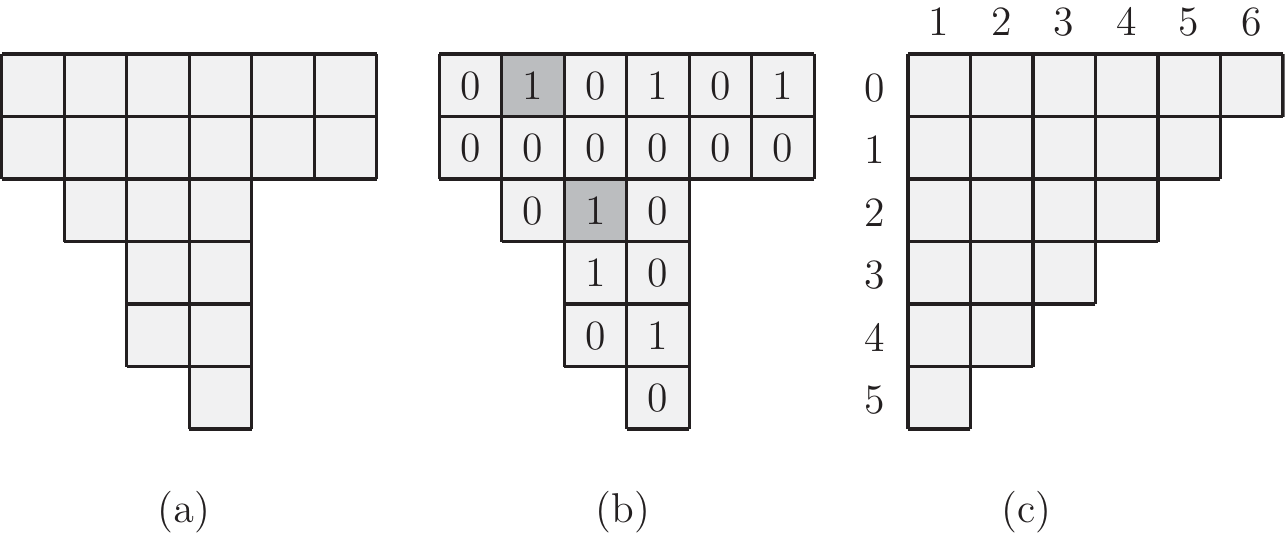}}
\caption{\small{ (a) a stack polyomino of shape $(2,3,5,6,2,2)$; (b) a filling of this polyomino, with a shaded $2$-diagonal; (c) the polyomino $A_7$.}}\label{polyominoes1}
\end{figure}

Let $A_n$ denote the polyomino $\{(i,j)\;|\;0\le i\le n-2,\; 1\le j\le n-i-1\}$ (see Fig.~\ref{polyominoes1}). 
That maximal $k$-diagonal-free fillings of the triangular polyomino $A_n$ are enumerated by the determinant of Theorem~\ref{triangenumeration} is proved in \cite{ht-gbmdpi-92}.

\end{enumerate}

Let us now give a little overview of the different proofs of Theorem~\ref{triangenumeration}:

\begin{enumerate}
\item First, Jonsson gives in~\cite{j-gt-03} a direct, but complicated, proof using induction on $n$, based in an analogue of our Corollary~\ref{enumerationinduction}.

\item In~\cite{j-gtdfssp-05}, Jonsson gives another proof, this time indirect. He observes that given any subset $E$ of $E_n$, an $\ell$-crossing of $E$ corresponds to an $\ell$-diagonal in the upper triangular part of the adjacency matrix of $E$. In particular, this gives a straightforward bijection between $k$-triangulations of an $n$-gon and maximal $k$-diagonal-free fillings of another polyomino, namely
$$\Omega_n:=\{(i,j)\;|\; 0\le i\le n-2,\; i+1\le j\le n-1\}.$$

\begin{example}
\rm
In Figure~\ref{polyominoes2}, we show the filling of $\Omega_8$ that corresponds to the $2$-triangulation of the octagon of Figure~\ref{fig:2triang8points}. For the sake of clarity, we have omitted $0$'s and replaced $1$'s by dots. We also have shaded the $2$-relevant boxes of $\Omega_8$.
\end{example}

\noindent Jonsson then proves that the number of maximal $k$-diagonal-free fillings of a stack polyomino of shape $(s_1,\ldots,s_r)$ depends only on its \emph{content}, that is, the multiset $\{s_1,\ldots,s_r\}$. Thus, he can use the polyomino $A_n$ instead of $\Omega_n$, and the aforementioned result of~\cite{ht-gbmdpi-92}, to enumerate $k$-triangulations.

\item Krattenthaler~\cite{k-gdidcffs-06}, studying different restrictions on diagonals in fillings of polyominoes, proves in a simpler way that the number of fillings of a stack polyomino only depends on its content. Rubey~\cite{r-idsfmp-07} generalizes this same result to moon polyominoes. He also emphasizes that ``the problem of finding a completely bijective proof [...] remains open. However, it appears that this problem is difficult".

\item Elizalde~\cite{e-btdp-06} gives an explicit bijection between $2$-triangu\-lations and Dyck $2$-paths. He begins with coloring the relevant edges of a $2$-triangulation with two colors. Then, each color defines one of the two Dyck paths.

\end{enumerate}

\begin{figure}
\centerline{\includegraphics[scale=1]{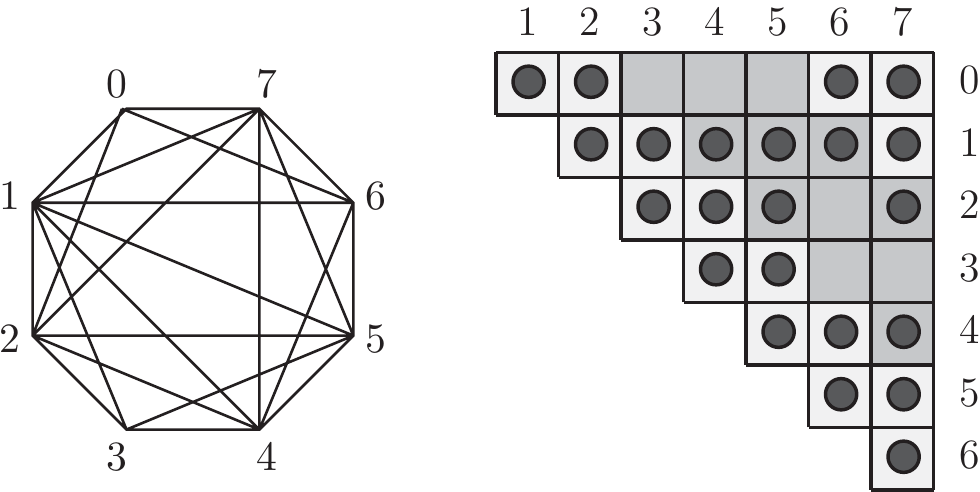}}
\caption{\small{The filling of $\Omega_8$ that corresponds fo the $2$-triangulation of Figure~\ref{fig:2triang8points}.}}\label{polyominoes2}
\end{figure}

Our hope would be that looking at $k$-triangulations as complexes of $k$-stars might perhaps be used to give a general explicit bijection between $k$-triangulation and Dyck $k$-paths. Observe that the number $n-2k$ of $k$-stars equals the semi-length of the Dyck $k$-paths we need to construct. Hence, the idea would be to use each $k$-star to represent one pair of steps in each path.

\medskip

To close this discussion, remember that Catalan numbers also count many other ``Catalan structures" (see Exercise~6.19 in~\cite{s-ec2-99}). Routed binary trees seem particularly important for our topic since they are dual graphs of triangulations.
It would be interesting to understand the corresponding objects for $k$-triangulations.

\bigskip

\paragraph{{\sc Rigidity}}

A triangulation of a convex $n$-gon is a \emph {minimally rigid} object in the plane. That is to say, any continuous movement of its vertices that preserves all edges lengths extends to an isometry of the plane, and the triangulation is a minimal graph for this property. Moreover, it is generically so, meaning that the same is true for all embeddings of the same graph with sufficiently generic choice of positions for the vertices.
See~\cite{g-cf-01} for a very nice expository introduction to combinatorial rigidity, and~\cite{gss-cr-93} for a more technical one.

In the plane, generic minimal rigidity of a graph $G=(V,E)$ is equivalent to the following \emph{Laman property}: $|E|=2|V|-3$ and any subgraph on $2\le m\le |V|$ vertices has less than $2m-3$ edges. The Laman property is a special case of the following:

\begin{definition}
A graph $G=(V,E)$ is said to be \emph{$(p,q)$-sparse} if $|E'|\le p|V(E')|-q$ for any subset $E'$ of $E$. 
\end{definition}

Sparsity for several families of parameters $(p,q)$ is relevant in rigidity theory, matroid theory and pebble games. See~\cite{ls-pgasg-07,st-shpga-07} and the references therein. Observe that Corollary~\ref{starsenumeration} implies:

\begin{corollary}
 $k$-triangulations are $\left(2k,{2k+1 \choose 2}\right)$-sparse graphs.
\end{corollary}

Being $\left(d,{d+1 \choose 2}\right)$-sparse is a necessary, but not sufficient for $d>2$, condition for a graph being generically minimally rigid in dimension $d$~\cite{gss-cr-93}. This suggests the following
conjecture:

\begin{conjecture}
Every $k$-triangulation is generically minimally rigid in dimension $2k$.
\end{conjecture}

We can prove this conjecture for $k=2$:

\begin{theorem}\label{rigid}
Every $2$-triangulation is a generically minimally rigid graph in dimension $4$.
\end{theorem}

\begin{proof}
We prove by induction on $n$ that $2$-triangulations are generically rigid in dimension $4$. Minimality then follows from the fact that a generically rigid graph in dimension $d$ needs to have at least $nd-{d+1 \choose 2}$ edges.

Induction begins with the unique $2$-triangulation on $5$ points, that is,  the complete graph $K_5$, which is generically rigid in dimension $4$.

For the inductive step, 
let us recall the following graph-theoretic construction called ``vertex split''
(see Fig.~\ref{split} and~\cite{w-vsif-90} for more details). Let $G$ be a graph, $u$ a vertex of $G$ and $U$ the adjacent vertices of $u$. Let $U=X\sqcup Y\sqcup Z$ be a partition of $U$. A \emph{vertex split} of $u$ on $|Y|$ edges is the graph $G'$ obtained from $G$ as follows:
\begin{enumerate}[(i)]
\item remove the edges $[u,z]$, for $z\in Z$;
\item add a new vertex $v$;
\item add new edges $[u,v]$, $[v,y]$, for $y\in Y$ and $[v,z]$, for $z\in Z$.
\end{enumerate}

\begin{figure}
\centerline{\includegraphics[scale=1]{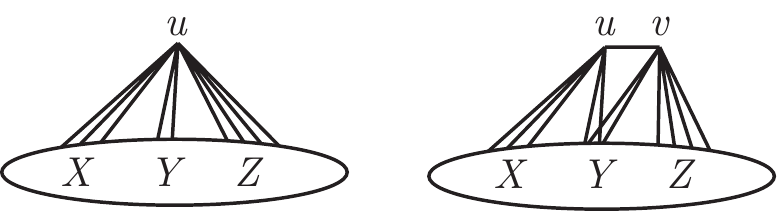}}
\caption{\small{A vertex split on $3$ edges.}}\label{split}
\end{figure}

The key result that we need is that a vertex split on $d-1$ edges in a generically rigid graph in dimension $d$ is also generically rigid in dimension $d$~\cite{w-vsif-90}.


So, let $n\ge 5$ and assume that we have already proved that every $2$-triangulation of the $n$-gon is rigid. Let $T$ be a $2$-triangulation of the $(n+1)$-gon. Let $S$ be a $2$-star of $T$ with at least two $2$-boundary edges (such a $2$-star exists since it appears in the ``outer" side of any $2$-ear). It is easy to check that inverse transformation of the flattening of $S$ is exactly a vertex split on $3$ edges. Thus, the result follows.
\end{proof}

Generalizing this proof, observe that if $T$ is a $k$-triangulation and $S$ a $k$-star of $T$ with $k$ $k$-boundary edges (or equivalently $k-1$ consecutive $k$-ears), then the inverse transformation of the flattening of $S$ is exactly a vertex split on $2k-1$ edges. However, our proof of Theorem~\ref{rigid} can not be directly applied since there exist $k$-triangulations with no $k$-star containing $k$ $k$-boundary edges, or equivalently, without $k-1$ consecutive $k$-ears, for $k\ge 3$ (see Fig.~\ref{ctrexm3}).

\begin{figure}
\centerline{\includegraphics[scale=1]{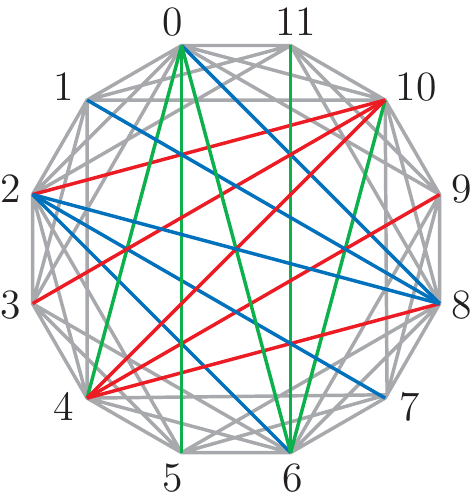}}
\caption{\small{A $3$-triangulation without $2$ consecutives $3$-ears.}}\label{ctrexm3}
\end{figure}

\bigskip

\paragraph{{\sc Multi-associahedron}}

Let $\Delta_{n,k}$ be the simplicial complex of all subsets of $k$-relevant edges of $E_n$ that do not contain any $(k+1)$-crossing.
Facets and ridges in $\Delta_{n,k}$ are $k$-triangulations of the $n$-gon and flips between them. Corollary~\ref{starsenumeration} proves that $\Delta_{n,k}$ is pure, and Corollary~\ref{graphflips} suggests that it could be a sphere; this was proved in~\cite{j-gt-03,j-gtdfssp-05}. But it remains a main open question to know whether $\Delta_{n,k}$ is polytopal. That is, whether it is the boundary complex of a polytope.

\begin{figure}[!h]
\centerline{\includegraphics[scale=.9]{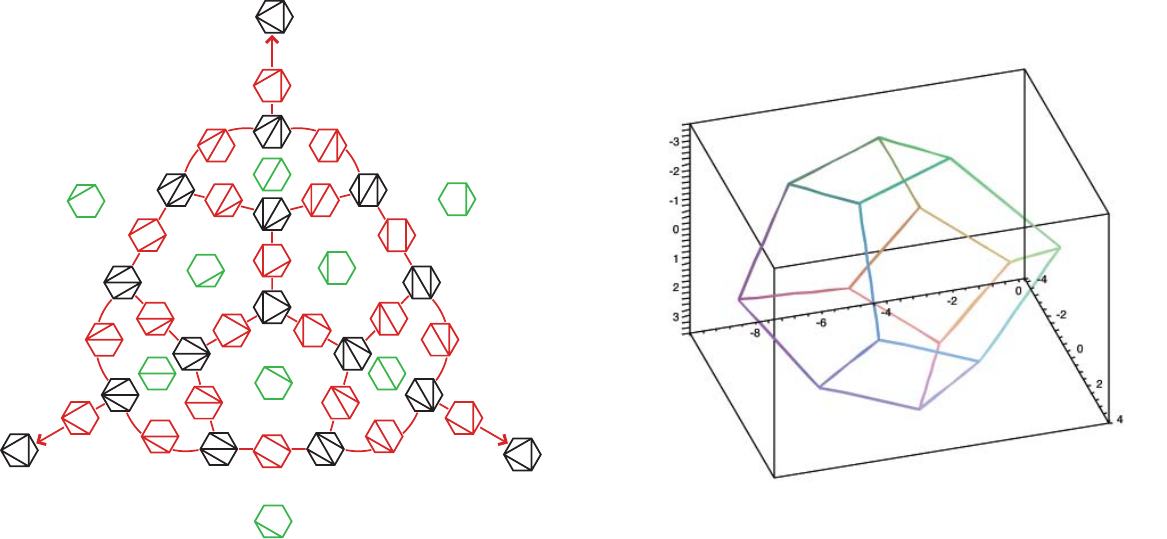}}
\caption{\small{The simplicial complex $\Delta_{6,1}$ and a realization of the corresponding associahedron.}}\label{associahedron}
\end{figure}

So far, we only know that this holds for:
\begin{itemize}
\item $k=1$: $\Delta_{n,1}$ is the boundary complex of the polar of the associahedron (see Fig.~\ref{associahedron}). Various realizations have been proposed; among others see~\cite{bfs-ccsp-90, fr-ga-05, hl-rac-07, l-at-89, l-rsp-04}.

\item $n=2k+1$: there is a unique $k$-triangulation.

\item $n=2k+2$: there are $k+1$ $k$-triangulations (Example~\ref{exm:2k+2}) 
and $\Delta_{2k+2,k}$ is realized by the $k$-simplex. 

\item $n=2k+3$:

\begin{lemma}
\label{lemma:2k+3}
For any $k\ge 1$, the simplicial complex $\Delta_{2k+3,k}$ is realized by the cyclic polytope of dimension $2k$ with $2k-3$ vertices.
\end{lemma}

\begin{proof}
The set of $k$-relevant edges of the $(2k+3)$-gon forms a cycle $C$ of length $2k+3$ (in fact, it forms a $(2k+3)$-star). It is easy to check that the sets of $k$-relevant edges of $k$-triangulations of the $(2k+3)$-gon are exactly the subsets of $C$ that satisfy Gale's evenness condition, and thus correspond to facets of the cyclic polytope (see Theorem 0.7 of~\cite{z-lp-95}).
\end{proof}
\end{itemize}

In our opinion, $k$-stars are a promising tool to answer this question: first because they give a better understanding of flips in $k$-triangulations and some constructions of the associahedron are based on understanding flips; but moreover because one of the most natural ways of constructing the associahedron is as a secondary polytope, which obviously uses triangles~\cite{bfs-ccsp-90, s-gbfscc-06}.

Another possible way of constructing this polytope could be based on rigity of $k$-triangulations, mimicking the construction in~\cite{rss-empppt-03} (see also~\cite{rss-pt-06}).

\bigskip

\paragraph{{\sc Surfaces}}

Regarding a $k$-triangulation $T$ of the $n$-gon as a complex of $k$-stars naturally defines a polygonal complex $\mathcal{C}(T)$ as follows:
\begin{enumerate}[(i)]
\item the vertices of $\mathcal{C}(T)$ are the vertices of the $n$-gon;
\item the edges of $\mathcal{C}(T)$ are the $k$-boundary edges and $k$-relevant edges of $T$;
\item the facets of $\mathcal{C}(T)$ are the $k$-stars of $T$, considered as $(2k+1)$-gons.
\end{enumerate}
Since every $k$-relevant edge belongs to two $k$-stars and every $k$-boundary edge belongs to one, this complex is a surface with boundary, with $\gcd(n,k)$ boundary components. Also, it is orientable since the natural orientation of each $k$-star can be inherited on each polygon. Hence, from its Euler characteristic, we derive its genus:
$$g_{n,k}=\frac{1}{2}(2-f+e-v-b)=\frac{1}{2}(2-n+k+kn-2k^2-\gcd(n,k)).$$
That is, the surface does not depend on the $k$-triangulation $T$ but only on $n$ and $k$. We denote $\mathcal{S}_{n,k}$ this surface. The polygonal complex $\mathcal{C}(T)$ of each $k$-triangulation provides a polygonal decomposition of $\mathcal{S}_{n,k}$.

Figure~\ref{surfaces} shows the surfaces associated to the $2$-triangulations $T_{n,2}^{\min}$ for $n=6,7$ and $8$.

\begin{figure}
\centerline{\includegraphics[scale=.7]{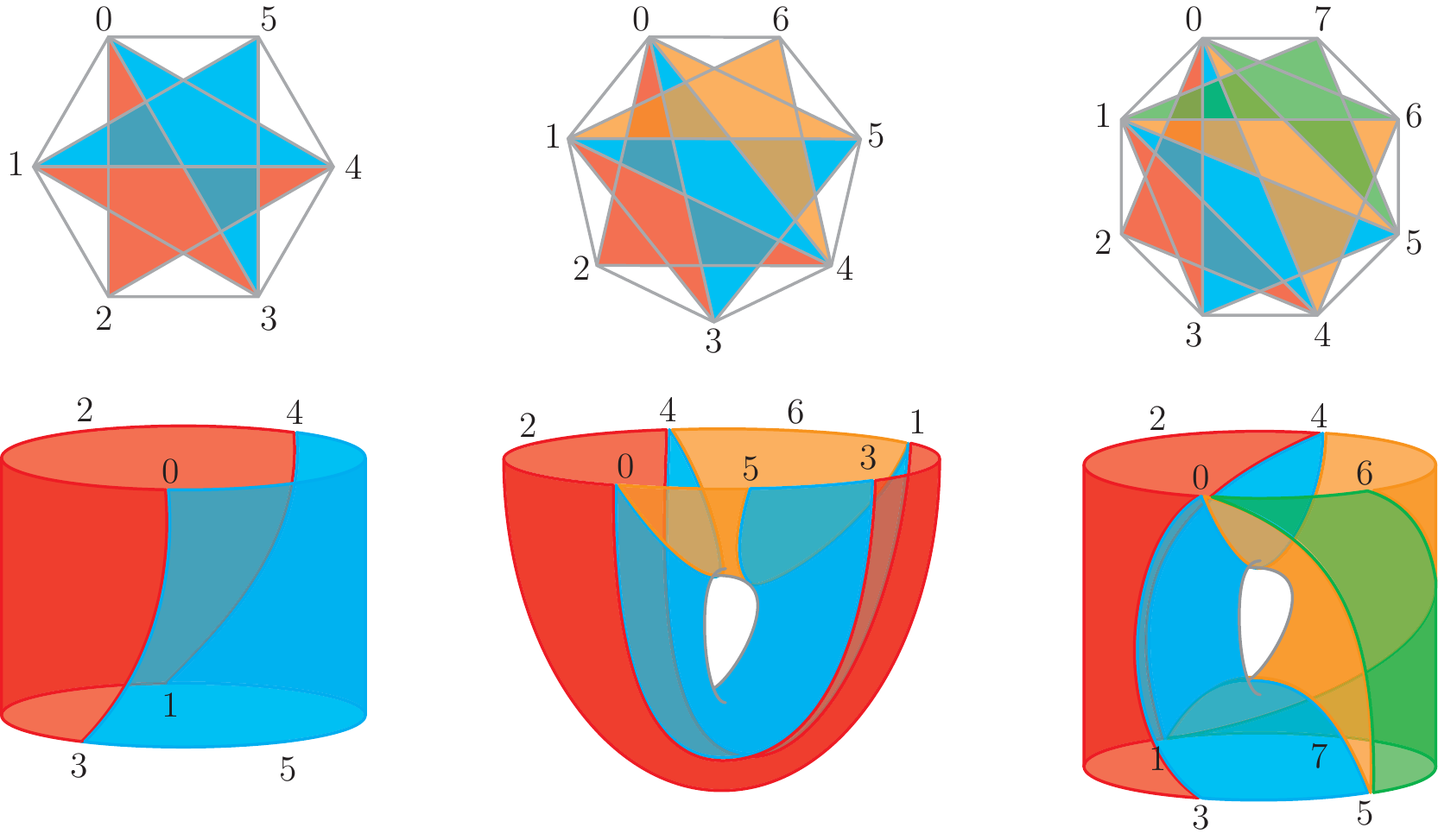}}
\caption{\small{Examples of decompositions of surfaces associated to the $2$-trangulations $T_{6,2}^{\min}$, $T_{7,2}^{\min}$ and $T_{8,2}^{\min}$.}}\label{surfaces}
\end{figure}

Apart from its beauty, this interpretation of $k$-triangulations as poly\-gonal decompositions of surfaces is interesting for the following reason.

Let $T$ be a $k$-triangulation of the $n$-gon, and let $f$ be a $k$-relevant edge of $T$. Let $R$ and $S$ be the two $k$-stars of $T$ containing $f$, and let $e$ be the common bisector of $R$ and $S$. Let $X$ and $Y$ be the two $k$-stars of $T\triangle\{e,f\}$ containing $e$.

Then $T\setminus\{f\}$ can be viewed as a decomposition of $\mathcal{S}_{n,k}$ into $n-2k-2$ $(2k+1)$-gons and one $4k$-gon, obtained from $\mathcal{C}(T)$ by gluing the two $(2k+1)$-gons $R$ and $S$ along the edge $f$. And then $T\triangle\{e,f\}$ is obtained from $T\setminus\{f\}$ by splitting this $4k$-gon into the two $(2k+1)$-gons $X$ and $Y$.

As an example, in Figure~\ref{monodromy} we have drawn the decomposition of the cylinder corresponding to the $2$-triangulation $T_{6,2}^{\min}$. The second $2$-triangulation is obtained from $T_{6,2}^{\min}$ by flipping the edge $[1,4]$, and we have represented the decomposition of the cylinder obtained by flipping the edge on the surface. If we now flip $[0,3]$ and then $[2,5]$, we obtain again the triangulation $T_{6,2}^{\min}$. 

\begin{figure}
\centerline{\includegraphics[scale=.7]{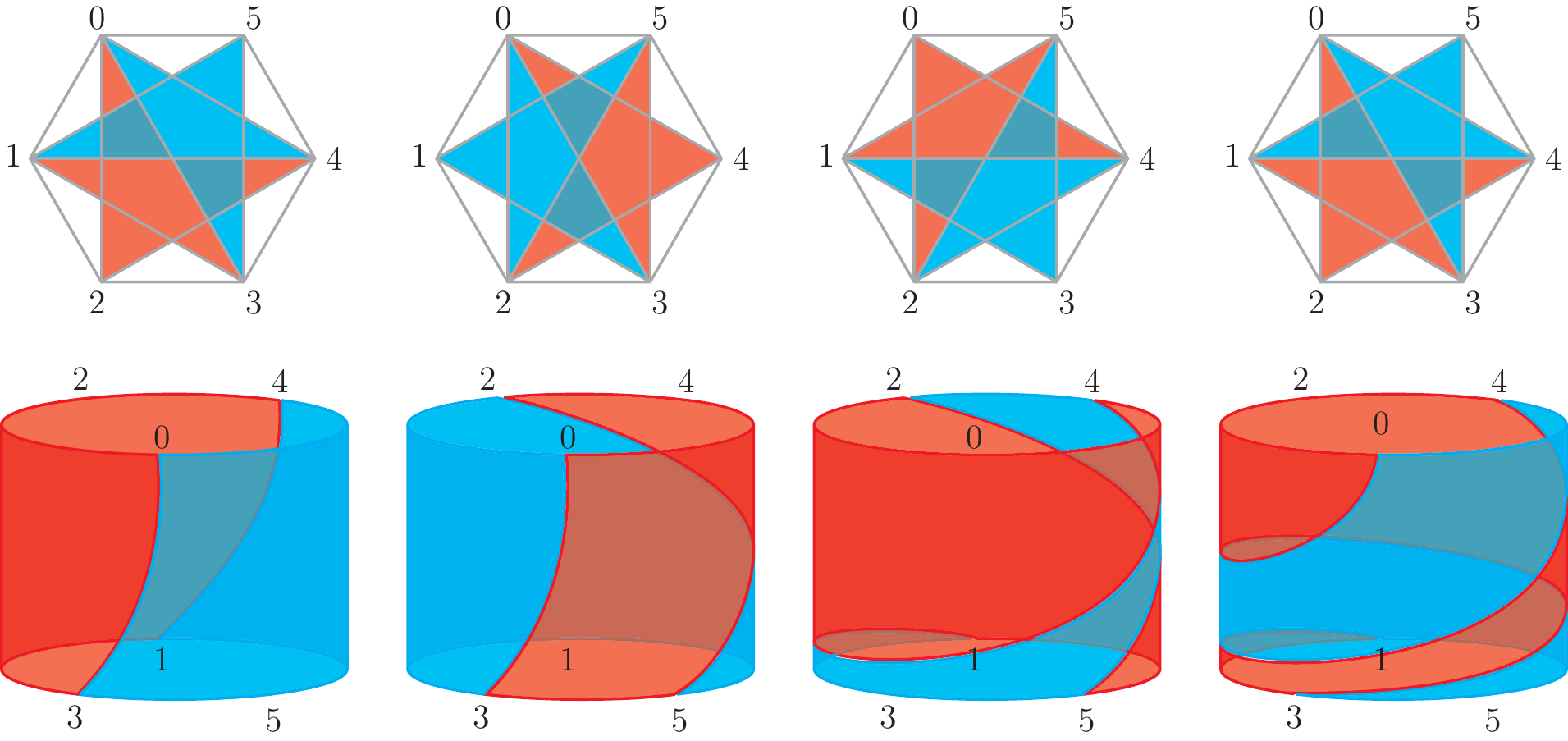}}
\caption{\small{A cycle of flips.}}\label{monodromy}
\end{figure}

Observe that although the corresponding decomposition is combinatorially the same as in the first picture, it has been ``twisted" on the surface. This phenomenon can be interpreted as a homomorphism between
\begin{enumerate}[(i)]
\item the \emph{fundamental group} $\pi_{n,k}$ of the graph of flips $G_{n,k}$ ({\it i.e.}~the set of closed walks in $G_{n,k}$, based at an initial $k$-triangulation, considered up to homotopy), and
\item the \emph{mapping class group} $\mathcal{M}_{n,k}$ of the surface $\mathcal{S}_{n,k}$ ({\it i.e.}~the set of diffeomorphisms of the surface $\mathcal{S}_{n,k}$ into itself that preserve the orientation and that fix the boundary of $\mathcal{S}_{n,k}$, considered up to isotopy~\cite{b-blmcg-74,i-mcg-02}).
\end{enumerate}
It may be interesting to understand the image and the kernel of this homomorphism. In particular, if this homomorphism is surjective, this interpretation provides a new combinatorial description of the mapping class group of $\mathcal{S}_{n,k}$.

This homomorphism is actually a \emph{monodromy action}. Let $\mathcal{G}_{n,k}$ denote the set of embeddings $f:T\to \mathcal{S}_{n,k}$, where $T$ ranges over all $k$-triangulations of the $n$-gon, and embeddings are considered modulo isotopy. We make
$\mathcal{G}_{n,k}$ into a graph, lifting flips between $k$-triangulations as in the example above. Then, the ``forgetful map'' is a covering map
$
\mathcal{G}_{n,k} \to G_{n,k}.
$
The mapping class group is the group of deck transformations of this cover and the above homomorphism is the corresponding monodromy action.

\bigskip

\paragraph{{\sc Chirotopes.}}

Let $T$ be a $k$-triangulation and let $S$ be a $k$-star of $T$. 
If we consider our $n$-gon with vertices in the unit circle $C\subset \mathbb{R}^2$, then lines that intersect $C$ form a M\"obius strip in the dual $\mathbb{R}P^2$. It is easy to check that the set of all geometric bisectors of a $k$-star $S$ (including the edges of $S$ as limit cases) form a \emph{pseudoline} $\ell_S$ in the M\"obius strip. Furthermore, as Michel Pocchiola pointed out to us, Theorem~\ref{common bisector} affirms that the two pseudolines corresponding to two $k$-stars of $T$ intersect exactly once. Thus, the set of pseudolines associated to the $k$-stars of $T$ is an arrangement of pseudolines in $\mathbb{R}P^2$. In particular, it defines a \emph{chirotope} or oriented matroid (of rank $3$) $\chi_T$ on the set of $k$-stars of $T$~\cite{blvswz-om-92,k-ah-92,rgz-om-04}. This is an analogue to the chirotope (or pseudo-line arrangement) that Pocchiola and Vegter introduce for pseudo-triangulations~\cite{pv-ot-94,pv-vc-96,rss-pt-06}.

We can also define the chirotope directly (that is, in the ``primal space") as follows. Let $S,S'$ and $S''$ be three $k$-stars of $T$. Let $u$ and $v$ be vertices of $S$ and $S'$ respectively such that $[u,v]$ is the unique common bisector of $S$ and $S'$. We define $\chi_T(S,S',S'')$ to be negative if $S''$ has more vertices in $\rrbracket u,v\llbracket$ than in $\rrbracket v,u\llbracket$, and positive in the other case (it is easy to see that $S''$ can not have has many vertices in $\rrbracket u,v\llbracket$ and in $\rrbracket v,u\llbracket$).

This chirotope provides a new interpretation of some of the results in this paper. For example, external $k$-stars ({\it i.e.}~$k$-star containing at least one $k$-boundary edge) are exactly the elements on the convex hull of this chirotope. Thus, Corollary~\ref{earsenumeration} of Section~\ref{sectionears} can be rewritten as: the number of $k$-ears of a $k$-triangulation $T$ of the $n$-gon is $n$ minus the number of $k$-stars of the convex hull of $\chi_T$. Remember also that external $k$-stars where relevant in Section~\ref{sectionflatinflat}. More technical details and further results on this topic will be given in a subsequent paper.


\section*{Acknowledgement}

We would like to thank Olivier Bernardi, Sergi Elizalde, Carsten Lange, Julian Pfeifle, Michel Pocchiola and Martin Rubey for interesting discussions and input, especially on the topics of Section~\ref{sectionopen}.

We also would like to thank the \emph{Centre de Recerca Matem\`atica} (Bar\-celona, Spain) for inviting both of us to stay there in January 2007, when most of these discussions took place.

\bibliographystyle{alpha}
\bibliography{multi-triangulations.bib}

\end{document}